\documentclass[11pt,reqno]{amsart}

\usepackage[a4paper,margin=30mm,headheight=14pt]{geometry}
\usepackage{amsmath,amssymb,amsthm,mathtools,mathrsfs}
\usepackage{xcolor}
\usepackage{tikz}
\usetikzlibrary{arrows.meta,calc,patterns,positioning}
\usepackage{enumitem}
\usepackage{etoolbox}
\usepackage{placeins}
\usepackage{flafter}
\usepackage[T1]{fontenc}
\usepackage{lmodern}
\usepackage{microtype}
\usepackage{hyperref}
\usepackage{aliascnt}
\usepackage[nameinlink,capitalize,noabbrev]{cleveref}

\makeatletter
\patchcmd{\section}
  {\z@{.7\linespacing\@plus\linespacing}{.5\linespacing}}
  {\z@{1.35\linespacing\@plus.25\linespacing\@minus.1\linespacing}
       {.65\linespacing\@plus.1\linespacing}}
  {}{\ClassError{amsart}{Unable to set section spacing}{Check the amsart section definition.}}
\patchcmd{\subsection}
  {\z@{.5\linespacing\@plus.7\linespacing}{-.5em}}
  {\z@{.9\linespacing\@plus.2\linespacing\@minus.1\linespacing}
       {.4\linespacing\@plus.05\linespacing}}
  {}{\ClassError{amsart}{Unable to set subsection spacing}{Check the amsart subsection definition.}}
\patchcmd{\paragraph}
  {\z@\z@{-\fontdimen2\font}}
  {\z@{.6\linespacing\@plus.1\linespacing}{-\fontdimen2\font}}
  {}{\ClassError{amsart}{Unable to set paragraph spacing}{Check the amsart paragraph definition.}}
\makeatother

\hypersetup{
  colorlinks=true,
  linkcolor=blue!45!black,
  citecolor=blue!45!black,
  urlcolor=blue!45!black,
  pdftitle={Cumulative angular cancellation in Maz'ya's endpoint inequalities},
  pdfauthor={Zhouyu Long and Wenming Zou}
}

\numberwithin{equation}{section}

\newtheorem{theorem}{Theorem}[section]
\newaliascnt{proposition}{theorem}
\newtheorem{proposition}[proposition]{Proposition}
\aliascntresetthe{proposition}
\newaliascnt{lemma}{theorem}
\newtheorem{lemma}[lemma]{Lemma}
\aliascntresetthe{lemma}
\newaliascnt{corollary}{theorem}
\newtheorem{corollary}[corollary]{Corollary}
\aliascntresetthe{corollary}
\theoremstyle{definition}
\newaliascnt{definition}{theorem}
\newtheorem{definition}[definition]{Definition}
\aliascntresetthe{definition}
\theoremstyle{remark}
\newaliascnt{remark}{theorem}
\newtheorem{remark}[remark]{Remark}
\aliascntresetthe{remark}
\theoremstyle{plain}

\crefname{theorem}{Theorem}{Theorems}
\Crefname{theorem}{Theorem}{Theorems}
\crefname{proposition}{Proposition}{Propositions}
\Crefname{proposition}{Proposition}{Propositions}
\crefname{lemma}{Lemma}{Lemmas}
\Crefname{lemma}{Lemma}{Lemmas}
\crefname{corollary}{Corollary}{Corollaries}
\Crefname{corollary}{Corollary}{Corollaries}
\crefname{definition}{Definition}{Definitions}
\Crefname{definition}{Definition}{Definitions}
\crefname{remark}{Remark}{Remarks}
\Crefname{remark}{Remark}{Remarks}

\newcommand{\R}{\mathbb R}
\newcommand{\Sph}{\mathbb S}
\newcommand{\supp}{\operatorname{supp}}
\newcommand{\dist}{\operatorname{dist}}

\newcommand{\sgn}{\operatorname{sgn}}
\newcommand{\cM}{\mathcal M}

\newcommand{\cD}{\mathcal D}

\newcommand{\cQ}{\mathcal Q}

\newcommand{\dd}{\,\mathrm d}
\newcommand{\eps}{\varepsilon}
\newcommand{\one}{\mathbf 1}
\newcommand{\wtK}{\widetilde K}
\newcommand{\Mp}{M_p}

\definecolor{modelblue}{RGB}{220,232,242}
\definecolor{modelgold}{RGB}{244,232,196}
\tikzset{
  geometric boundary/.style={draw=black,line width=.9pt},
  selected branch/.style={draw=black,line width=1.35pt},
  assigned cube/.style={draw=black,line width=.8pt},
  enlarged cube/.style={draw=black,densely dashed,line width=.7pt,rounded corners=3pt},
  kernel annulus/.style={draw=black,densely dotted,line width=.8pt},
  construction line/.style={draw=black!65,dash pattern=on 3pt off 1.5pt,line width=.55pt},
  inward normal/.style={-{Stealth[length=2.2mm,width=1.5mm]},line width=.75pt},
  outward normal/.style={-{Triangle[open,length=2.2mm,width=1.7mm]},line width=.75pt},
  map arrow/.style={-{Stealth[length=2.4mm,width=1.6mm]},line width=.75pt},
  figure label/.style={font=\scriptsize,inner sep=1.2pt,fill=white,fill opacity=.88,
    text opacity=1},
  panel title/.style={font=\footnotesize\bfseries,anchor=west},
  figure point/.style={circle,fill=black,inner sep=1.25pt}
}

\title[Angular cancellation]{Cumulative angular cancellation in Maz'ya's endpoint inequalities}
\author[Zhouyu Long]{\bf Zhouyu Long$^{*}$}
\author[Wenming Zou]{\bf Wenming Zou$^{\dagger}$}
\thanks{$^{*}$Department of Mathematical Sciences, Tsinghua University,
Beijing 100084, China.\newline
E-mail: \href{mailto:longzy25@mails.tsinghua.edu.cn}{\mbox{\nolinkurl{longzy25@mails.tsinghua.edu.cn}}}}
\thanks{$^{\dagger}$Department of Mathematical Sciences, Tsinghua University,
Beijing 100084, China.\newline
E-mail: \href{mailto:zou-wm@mail.tsinghua.edu.cn}{\mbox{\nolinkurl{zou-wm@mail.tsinghua.edu.cn}}}}

\makeatletter
\patchcmd{\@setauthors}{\centering\footnotesize}{\centering\large}{}%
  {\ClassError{amsart}{Unable to set author font}{Check the amsart author layout.}}
\patchcmd{\@setauthors}{\MakeUppercase{\authors}}{\authors}{}%
  {\ClassError{amsart}{Unable to preserve author case}{Check the amsart author layout.}}
\def\author@andify{\nxandlist{\unskip,\quad}{\unskip,\quad}{\unskip,\quad}}
\let\addresses\@empty
\makeatother
\date{}


\begin{document}
\raggedbottom
\begin{abstract}
We establish domain criteria for critical homogeneous potential
inequalities in the plane. Under signed plane cancellation condition, a uniform
absolute sum of frozen angular moments at each boundary centre implies
the nonlinear endpoint estimate
on any bounded open set. For quadratic Newton kernel, uniform
boundedness of signed cumulative moments over all source centres is
necessary and sufficient; one radial test family suffices. Tangent-cone
cancellation on finitely cornered Dini domains follows as a geometric
corollary. On one explicit strictly convex $C^1$, non-Dini domain, the
estimate holds for $2\xi_1\xi_2$ and fails for
$\xi_1^2-\xi_2^2$ with loglog divergence. A Laplacian-measure extension
transfers density estimates to PDE inequalities on the Dini class and on
strictly convex $C^1$ domains. We also classify quadratic cancellation on
polygons. The signed converse proved here is restricted to quadratic
integrands.
\end{abstract}
\maketitle

\section{Introduction and main results}
\label{sec:introduction}

Maz'ya's critical endpoint inequalities control signed nonlinear
integrals although the corresponding strong $L^1$ potential estimate
fails. Their whole-space form is
\begin{equation}\label{eq:intro-mazya}
 \left|\int_{\R^d}\Phi(\nabla u)\,\dd x\right|^{(d-1)/d}
 \lesssim\|\Delta u\|_{L^1(\R^d)},\qquad u\in C_c^\infty(\R^d),
\end{equation}
where $d\ge2$ and $\Phi$ is positively $d/(d-1)$-homogeneous,
Lipschitz on the sphere, and has vanishing spherical average
\cite{Mazya2010,MazyaProblems,MazyaShaposhnikova}.
Stolyarov proved the whole-space convolution theorem
\cite{StolyarovWhole} and the bounded $C^{1,\beta}$ and half-space
criteria \cite[Theorems~1.2 and~1.3]{StolyarovDomains}.
His Corollary~1.1 gives the smooth-domain Newton PDE estimate.
The present paper separates the role of boundary tangent models from
the accumulation of the angular error that remains between scales.

\paragraph{The domain-sensitive principle.}
For the critical homogeneous kernel $K$ and nonlinearity $\Phi$ specified
below, let $K_n$ be its sharp dyadic annular pieces and set
\[
 \mathfrak a_{n,\Omega}(a)
 =\max_{\sigma=\pm1}\left|\int_\Omega
                      \Phi(\sigma K_n(x-a))\,\dd x\right|,
 \qquad
 \mathcal B_{K,\Phi}(\Omega)
 =\sup_{a\in\partial\Omega}\sum_n\mathfrak a_{n,\Omega}(a).
\]
Our first general result, \cref{thm:atomic-criterion}, says that signed
plane cancellation and $\mathcal B_{K,\Phi}(\Omega)<\infty$ imply
\[
 \left|\int_\Omega\Phi(K*f)\,\dd x\right|
 \le C_{K,\Phi,p}(1+\mathcal B_{K,\Phi}(\Omega))\|f\|_1^p
\]
on every nonempty bounded open set. No smoothness or existence of
tangent models is assumed here. The supremum is outside the scale sum:
different boundary points may have different bad scales. The proof
freezes each localized source, transfers its centre to one
scale-independent approximate nearest-boundary projection, and pays
the transfer error with a fixed-closed-set moment estimate.

For a quadratic Newton integrand $\Phi_A(\xi)=\xi^{\mathsf T}A\xi$,
$A=A^{\mathsf T}$ and $\operatorname{tr}A=0$, there is an exact signed
criterion on arbitrary bounded open sets:
\[
 \sup_{a\in\R^2,\ \varepsilon>0}
 \left|\int_{\Omega\setminus B(a,\varepsilon)}
             \Phi_A(K(x-a))\,\dd x\right|<\infty,
 \qquad K(x)=\frac{x}{2\pi|x|^2}.
\]
By \cref{thm:quad-cumulative-criterion}, this condition is equivalent
to the full signed-data density inequality. Testing all translates
and dilates of one radial unit-mass mollifier already suffices. A
near--far comparison of the bilinear two-source kernel proves the
reverse implication. Unlike $\mathcal B$, this criterion permits
cancellation between scales. Its sufficiency is proved only for
quadratic integrands.

\paragraph{Tangent cones and the Dini corollary.}
On a finitely cornered planar domain whose boundary branches share a
Dini tangent modulus, signed cancellation on the plane, all regular
tangent half-planes, and the full vertex cones verifies the atomic
budget. This gives the density criterion
\cref{thm:curved-density} throughout $0<\alpha<2$,
$p=2/(2-\alpha)$. Necessity holds on the larger piecewise-$C^1$ class.
The infinite-sector version \cref{thm:sector-density} treats compactly
supported mean-zero data with constants independent of support location
and radius.

The full vertex cone is essential. For $0<\theta<2\pi$,
$\theta\ne\pi$, put
\begin{equation}\label{eq:intro-vertex-form}
 \Phi_\theta(\xi)
 =(\xi\cdot e_{\theta/2})^2-(\xi\cdot e_{\theta/2+\pi/2})^2,
 \qquad e_t=(\cos t,\sin t).
\end{equation}
Every half-circle moment vanishes, while
$\int_0^\theta\Phi_\theta(e_t)\,\dd t=\sin\theta\ne0$.
At scales meeting both boundary branches, their combined cone therefore
imposes a new constraint. Fixed-set summation keeps this model change
compatible with source-dependent minimizers.

\paragraph{One non-Dini domain, two outcomes.}
The explicit strictly convex $C^1$ domain in
\cref{eq:pair-domain} has a single non-Dini point with horizontal
tangent. On that same domain, the estimate holds for
$\Phi_+(\xi)=2\xi_1\xi_2$ but fails for
$\Phi_-(\xi)=\xi_1^2-\xi_2^2$; see \cref{thm:pair-nondini}.
The good angular moment is bounded uniformly by
$C(\log(e/r)^{-2}+r)$: the narrow angular defect is multiplied by
the integrand's zero on the rough tangent. The bad moment instead
accumulates a nonzero $\log\log(1/\varepsilon)$ term.
More generally, every nonzero admissible $\Phi$ has a fixed convex
$C^1$ counterexample with $N_\Omega(u_\varepsilon)\le2$;
see \cref{cor:convex-c1-pde-failure}. Thus ordinary Dini regularity is
a useful integrand-independent sufficient condition, but is not
necessary for a fixed nonzero domain--integrand pair.

\paragraph{The PDE realization.}
Write
$N_\Omega(u)=\|\Delta u\|_{L^1(\Omega)}
+\|\partial_nu\|_{L^1(\partial\Omega)}$.
On finitely cornered Dini domains, the tangent criterion is equivalent
to $|\int_\Omega\Phi(\nabla u)|^{1/2}\le C N_\Omega(u)$;
finitely many holes are allowed. The bridge is a linear potential
representation $K*\mu=\nabla u$ in $\Omega$, with compactly supported
$\mu$, zero mass and $\|\mu\|_{\cM}\le C_\Omega N_\Omega(u)$.
It also gives a constant-preserving extension with controlled scalar
Laplacian measure. Energy membership is qualitative, not bounded by
this $L^1$ budget.
\Cref{thm:convex-extension} supplies the same bridge on strictly convex
$C^1$ domains without a Dini assumption, using convex visual-angle
control and classical $L^2$ Neumann layer solvability.
Consequently the same-domain good/bad contrast holds for the full PDE
inequality as well; see \cref{cor:pair-pde}.

\paragraph{Methods, antecedents, and consequences.}
Localization follows \cite[equations~(4.11)--(4.13)]{StolyarovDomains}.
The fixed-closed-set moment theorem recasts its Lemmas~4.3--4.4;
the mixed estimate for $1<p<2$ is
\cite[Theorem~2.3]{StolyarovWhole}. Appendix~\ref{app:superquadratic-mixed}
gives a direct positive-kernel proof for $p\ge2$, a range already
discussed there. The new uses of this machinery are the cumulative
atomic criterion, the exact quadratic testing reduction, and the
domain--integrand distinction beyond Dini geometry.
The smooth-domain extension predecessor is
\cite[Proposition~5.1]{StolyarovDomains}; the endpoint Neumann background
is classical \cite{KralMedkova1998,MedkovaNeumann1998,Verchota1984}.
Our Dini construction uses one $L^1/L^2$ decomposition for the same
solution density. References to theorem and equation numbers in the
two Stolyarov papers use their linked author preprints.

Operator cancellation and boundary ellipticity govern related linear
endpoint inequalities \cite{BourgainBrezis,VanSchaftingen2008,
VanSchaftingen2013,GmeinederRaitaVanSchaftingen}; for domain and endpoint
context see \cite{GmeinederRaitaDomains,Spector}. Here the scalar moment
depends jointly on $K$, $\Phi$ and the domain. The tangent cone is used
for a nonlinear critical integral, rather than the linear singular
asymptotics of conic boundary problems
\cite{Kondratev1967,Grisvard,KozlovMazyaRossmann}.
As a consequence, \cref{thm:quadratic-polygon-moduli} classifies the
nonrigid polygon locus; among convex polygons it consists precisely
of parallelograms. Every fixed polygon nevertheless has an
infinite-dimensional smooth angular cancellation space.

The paper has six main sections. After the statements below,
\cref{sec:atomic-analysis} proves the abstract atomic criterion;
\cref{sec:tangent-geometry} verifies it geometrically and proves the
obstructions. \Cref{sec:pde-realization} gives the measure and PDE
arguments, \cref{sec:quadratic-results} treats quadratic testing and
polygonal consequences, and \cref{sec:scope} states the remaining gap.

\paragraph{\bf Notations.} \(B(a,r)\) is the open Euclidean ball and \(B_r:=B(0,r)\);
\(\ell(Q)\) is the side length of a square \(Q\);
\(\mathcal L^2\) and \(\mathcal H^1\) are Lebesgue and one-dimensional
Hausdorff measure; and \(\sigma:=\mathcal H^1\lfloor\Sph^1\).
For measurable $E\subset\R^2$, write $|E|:=\mathcal L^2(E)$.
The symbol \(\triangle\) denotes symmetric difference, and
\(L_c^s(D;E)\), \(1\leq s\leq\infty\), denotes the \(E\)-valued \(L^s\) functions
with compact support in \(D\). We write \(\one_E\) for the indicator of
\(E\); \(A\lesssim B\) means \(A\leq CB\) with \(C\) depending only on the
fixed parameters of the current statement, and \(A\simeq B\) means both
inequalities; a subscript records the permitted parameter dependence of the
implicit constants. The notation \(A\Subset B\) means that \(\overline A\)
is a compact subset of \(B\). We identify $\R^2$ with $\mathbb C$ whenever
complex notation is used, and set $\dist(E,\varnothing):=+\infty$.
For a bounded Lipschitz domain $\Omega$, $n$ denotes the
$\mathcal H^1$-almost-everywhere defined outward unit normal on
$\partial\Omega$; no value is assigned at corner points. We write
$C^\infty(\overline\Omega;\R)$ for restrictions to $\overline\Omega$ of
smooth functions on an open neighborhood of $\overline\Omega$.
For a locally compact space $X$ and a finite-dimensional Euclidean space
$E$, $\cM(X;E)$ is the space of finite $E$-valued Radon measures with its
total-variation norm, and $\cM(X):=\cM(X;\R)$. For a planar vector
distribution $G$, our convention is
$\operatorname{curl}G=\partial_1G_2-\partial_2G_1$.

\paragraph{\bf Constant convention.}
The analytic data $m,\alpha,p,K,\Phi$ and the domain or sector appearing in
the statement under consideration are fixed before any density, scale, or
truncation parameter is chosen. Unless a subscript says otherwise, implicit
constants may depend on these fixed data, on the displayed finite geometric
atlas, and on the fixed localization and adjacent-grid parameters. In the stated density, extension, and PDE inequalities, these constants
are independent of $s$, $f$, $u$, a dyadic scale $n$, a truncation threshold,
a support radius removed by scaling, and an approximation index.
For a Dini boundary, constants may depend on the fixed common tangent
modulus and its Dini integral, the finite boundary decomposition,
the corner angles, and the fixed cutoffs. No uniformity is asserted as
an angle tends to $0$ or $2\pi$ or the geometry degenerates.
Qualitative bounds for a fixed datum may depend on that datum.

\subsection{Analytic hypotheses}
\label{sec:kernels-domains}

Assume throughout that
\begin{equation*}\tag{$\mathrm H_{K,\Phi}$}\label{eq:standing-analytic-hypothesis}
 \begin{gathered}
  m\geq1\ \text{is an integer},\qquad 0<\alpha<2,\qquad
  p=\dfrac{2}{2-\alpha}\in(1,\infty),\\
  0<\gamma\leq1,\qquad \gamma=1\ \text{if }0<\alpha<1,\\
  K:\R^2\setminus\{0\}\to\R^m,\qquad
  K(x)=|x|^{\alpha-2}\wtK(x/|x|),\qquad
  \wtK\in C^{0,\gamma}(\Sph^1;\R^m),\\
  \Phi:\R^m\to\R,\qquad
  \Phi(tz)=t^p\Phi(z)\quad(t>0,\ z\in\R^m),\qquad
  \Phi|_{\Sph^{m-1}}\in\operatorname{Lip}(\Sph^{m-1}).
 \end{gathered}
\end{equation*}
Here $C^{0,\gamma}$ denotes angular H\"older continuity, with
$C^{0,1}=\operatorname{Lip}$. Positive homogeneity implies $\Phi(0)=0$;
no parity condition is imposed.
Since $0<\alpha<2$, $K$ is locally integrable; throughout we assign it the
value $K(0)=0$. This convention does not affect Lebesgue integrals or
the almost-everywhere convolution identities used below.

The mixed estimate used below is stated in \cref{thm:stolyarov-mixed}. Its
range $1<p<2$ is taken from \cite{StolyarovWhole}; its range $p\geq2$
is proved directly in Appendix~\ref{app:superquadratic-mixed}. Thus it is a theorem
under the standing assumptions, not an additional hypothesis.

\subsection{Cumulative criteria on arbitrary bounded open sets}
For the homogeneous kernel above, use the sharp annular pieces
$K_n=K\mathbf1_{\{2^{-n-1}\le|x|<2^{-n}\}}$. For a measurable set
$D\subset\R^2$, define
\begin{equation}\label{eq:atomic-annular-defect}
 \mathfrak a_{n,D}(a)
 :=\max_{\varsigma\in\{-1,1\}}
 \left|\int_D\Phi(\varsigma K_n(x-a))\,\dd x\right|.
\end{equation}

For a bounded open set $\Omega\ne\varnothing$, retain the nonnegative
two-sign annular moment $\mathfrak a_{n,\Omega}$ from
\eqref{eq:atomic-annular-defect} and put
\begin{equation}\label{eq:atomic-cumulative-budget}
 \mathcal B_{K,\Phi}(\Omega)
 :=\sup_{a\in\partial\Omega}
       \sum_{n\in\mathbb Z}\mathfrak a_{n,\Omega}(a).
\end{equation}
The absolute value is taken separately on each annulus.  Thus this
quantity permits angular cancellation within an annulus, but does not
use cancellation between different scales.  The supremum is outside the
scale sum.  No boundary regularity or existence of tangent models is
part of the definition.

\begin{theorem}\label{thm:atomic-criterion}
Assume \eqref{eq:standing-analytic-hypothesis} and signed plane
cancellation, and let $\Omega\subset\R^2$ be a nonempty bounded open set.
If $\mathcal B_{K,\Phi}(\Omega)<\infty$, then
\begin{equation}\label{eq:atomic-cumulative-annular}
 \sum_{n\in\mathbb Z}
 \left|\int_\Omega\Phi(K_n*f)\,\dd x\right|
 \le C_{K,\Phi,p}
      \bigl(1+\mathcal B_{K,\Phi}(\Omega)\bigr)\|f\|_1^p
 \qquad(f\in L_c^1(\R^2;\R)).
\end{equation}
Consequently, for every $f\in L^1(\R^2;\R)\cap
L^s_{\rm loc}(\R^2;\R)$, $1<s\le\infty$,
\begin{equation}\label{eq:atomic-cumulative-density}
 \left|\int_\Omega\Phi(K*f)\,\dd x\right|
 \le C_{K,\Phi,p}
      \bigl(1+\mathcal B_{K,\Phi}(\Omega)\bigr)\|f\|_1^p.
\end{equation}
The constants are independent of the open set except through its
displayed atomic budget.  In particular, the stronger condition
$\sum_n\sup_{a\in\partial\Omega}\mathfrak a_{n,\Omega}(a)<\infty$
also suffices.
\end{theorem}

For the Newton kernel $K(x)=x/(2\pi|x|^2)$, quadratic integrands admit
an exact criterion that allows cancellation between different scales.
Let $A=A^{\mathsf T}\in\R^{2\times2}$ with $\operatorname{tr}A=0$,
and put $\Phi_A(\xi)=\xi^{\mathsf T}A\xi$.
For a bounded open set $\Omega$, define
\begin{equation}\label{eq:quad-cumulative-test}
 \mathcal P_A(\Omega)
 :=\sup_{\substack{a\in\R^2\\\varepsilon>0}}
 \left|\int_{\Omega\setminus B(a,\varepsilon)}
                  \Phi_A(K(x-a))\,\dd x\right|.
\end{equation}
In particular, the supremum includes centers outside $\overline\Omega$.
If
\[
 \mathcal A_A(a,r)
 :=\int_{\Sph^1}\one_\Omega(a+r\omega)
                    \Phi_A(K(\omega))\,\dd\sigma(\omega),
\]
polar integration gives the equivalent expression
\begin{equation}\label{eq:quad-cumulative-angular}
 \mathcal P_A(\Omega)
 =\sup_{a,\varepsilon}
       \left|\int_\varepsilon^\infty
                      \mathcal A_A(a,r)\frac{\dd r}{r}\right|.
\end{equation}
The absolute value is taken after integration in $r$; no absolute
summability of the scale contributions is required.

\begin{theorem}
\label{thm:quad-cumulative-criterion}
Let $\Omega\subset\R^2$ be nonempty, bounded and open, and let $A$ be
as above. Then $\mathcal P_A(\Omega)<\infty$ if and only if there is
$C_A(\Omega)<\infty$ such that
\begin{equation}\label{eq:quad-cumulative-density}
 \left|\int_\Omega\Phi_A(K*f)\,\dd x\right|
 \leq C_A(\Omega)\|f\|_1^2
 \quad\bigl(f\in L^1(\R^2;\R)\cap L^s_{\rm loc}(\R^2;\R),\
                 1<s\leq\infty\bigr).
\end{equation}
For the optimal constant in \eqref{eq:quad-cumulative-density},
\[
 \mathcal P_A(\Omega)\leq C_A(\Omega)+C\|A\|_{\rm op},
 \qquad
 C_A(\Omega)\leq\mathcal P_A(\Omega)+C\|A\|_{\rm op},
\]
with a numerical constant $C$ independent of $A$ and $\Omega$.
No boundary regularity is assumed. Equivalently, it suffices to test
\eqref{eq:quad-cumulative-density} on all translates and dilates of
one fixed nonnegative radial function in $C_c^\infty(B_1)$ of integral one.
\end{theorem}

\subsection{Polygonal domains}\label{subsec:polygonal-domains}

Let $P\subset\R^2$ be a bounded connected Lipschitz polygon with finitely many boundary components; holes are allowed. Every boundary component has finitely many edges and vertices. We assume that the listed vertices are nonflat; redundant collinear vertices may always be removed. All polygonal tangent-model normals in this subsection are oriented into $P$, including on inner boundary components. Write
\(\mathcal E(P)\) for the set of open edges and \(\mathcal V(P)\) for the
vertex set. For $e\in\mathcal E(P)$, let $n_e^{\mathrm{in}}$ be the inward unit normal and define the origin-based tangent half-plane
\[
 H_e^0=\{z\in\R^2:z\cdot n_e^{\mathrm{in}}>0\}.
\]
Thus, for every $a$ in the relative interior of $e$, the local affine tangent model is $a+H_e^0$. At a vertex $v$, let $C_v^0$ be the origin-based tangent cone, so that the local affine model is $v+C_v^0$. After a rotation,
\[
 C_v^0=\{r e_t:r>0,\ 0<t<\theta_v\},
 \qquad e_t=(\cos t,\sin t),
 \qquad 0<\theta_v<2\pi,
\]
with $\theta_v\neq\pi$.

\subsection{Finitely cornered piecewise-\texorpdfstring{$C^{1,\mathrm{Dini}}$}{C1Dini} domains}

A bounded connected Lipschitz domain $\Omega$ is called a
\emph{finitely cornered piecewise-$C^{1,\mathrm{Dini}}$ domain} if there is a finite set
$\mathcal V(\Omega)\subset\partial\Omega$ such that
$\partial\Omega\setminus\mathcal V(\Omega)$ is the disjoint union of finitely many connected regular
embedded $C^1$ boundary pieces with the following uniform Dini regularity.
A piece is either an embedded closed $C^1$ curve, or an open arc with
a finite-length arclength parametrization that extends regularly in $C^1$ to both endpoints in
$\mathcal V(\Omega)$; the two limiting vertices are allowed to coincide.
There is a bounded nondecreasing function
$\omega:[0,\infty)\to[0,\infty)$, with $\omega(0)=0$ and
$\omega(r)>0$ for $r>0$, such that
\begin{equation}\label{eq:dini-tangent-modulus}
 \lim_{r\downarrow0}\omega(r)=0,\qquad
 \int_0^1\frac{\omega(r)}r\,\dd r<\infty,\qquad
 |\tau(s)-\tau(t)|\leq\omega(|s-t|).
\end{equation}
The last inequality holds for the unit tangent $\tau$ of every piece
in its arclength parametrization, including the one-sided endpoint
values. On a closed piece the arclength parametrization is periodic,
and the inequality is required whenever $|s-t|$ is at most half its
period. We fix one such common modulus $\omega$ for the domain.
At every $v\in\mathcal V(\Omega)$ exactly two one-sided $C^1$ branches meet, and their tangent
lines are distinct. Thus $v$ has a nondegenerate, nonflat interior tangent cone $C_v^0$.
Closed boundary components without corners, components with a single corner, and finitely many holes are allowed.
The piecewise-$C^{1,\beta}$ domains, $0<\beta\leq1$, are included:
their unit tangents satisfy \eqref{eq:dini-tangent-modulus} with
$\omega(r)=C\min\{r^\beta,1\}$ for a sufficiently large fixed $C$.

For the finitely cornered domains considered here, a \emph{noncorner boundary
point} is a point of $\partial\Omega\setminus\mathcal V(\Omega)$,
where the boundary has the stated $C^{1,\mathrm{Dini}}$ regularity.
For such a point \(a\), let \(n_a^{\mathrm{in}}\) be the inward
unit normal and put
\[
 H_a^0:=\{z\in\R^2:z\cdot n_a^{\mathrm{in}}>0\},
 \qquad
 d_\Omega(a):=\dist(a,\mathcal V(\Omega)).
\]
With the convention $d_\Omega(a)=+\infty$ when $\mathcal V(\Omega)=\varnothing$, the natural geometric
class above has the quantitative approximation used in the proof.

\begin{lemma}\label{lem:piecewise-model-bounds}
For every finitely cornered piecewise-$C^{1,\mathrm{Dini}}$ domain $\Omega$ with fixed modulus $\omega$
there are $r_0,C_0,c_0>0$ such that
\begin{align}
 |(\Omega\mathbin\triangle(a+H_a^0))\cap B(a,r)|&\leq C_0r^2\omega(2r)
 &&\text{noncorner }a,\quad 0<r<\min\{r_0,c_0d_\Omega(a)\},
 \label{eq:smooth-model}\\
 |(\Omega\mathbin\triangle(v+C_v^0))\cap B(v,r)|&\leq C_0r^2\omega(2r)
 &&v\in\mathcal V(\Omega),\quad 0<r<r_0.
 \label{eq:vertex-model}
\end{align}
Here $H_a^0$ and $C_v^0$ are the origin-based tangent half-plane and tangent cone; their affine
realizations at the boundary points are $a+H_a^0$ and $v+C_v^0$. The restriction
$r<c_0d_\Omega(a)$ prevents a comparison ball at a noncorner boundary point from crossing a corner; when the comparison involves a corner, the vertex model in \eqref{eq:vertex-model} is used instead.
For a piecewise-$C^{1,\beta}$ domain the right-hand sides are
$O(r^{2+\beta})$. In particular, a truncated sector belongs to this
class with $\beta=1$, and its errors are $O(r^3)$.
\end{lemma}

For $\Phi(K(\cdot-a))$, critical homogeneity turns the
$r^2\omega(2r)$ area bound into an $O(\omega(2r))$ integral error
on an annulus of radius $r$.
The Dini condition makes these errors summable over dyadic scales.

\begin{proof}
On a uniform arclength neighborhood of every regular point and
one-sided endpoint, tangent variation is below $1/2$. Rotation gives
a graph $y=h_a(x)$ with $h_a(0)=h_a'(0)=0$ and
\[
 |s-s_a|\leq2|x|,\qquad |h_a'(x)|\leq2\omega(2|x|),\qquad
 |h_a(x)|\leq2|x|\omega(2|x|).
\]
Choose disjoint vertex neighborhoods containing only their two
terminal branches. Transversality puts those branches in disjoint
narrow cones, so for a regular point $a$ near $v$ the other branch
has distance at least $c|a-v|=c\,d_\Omega(a)$. Away from these
neighborhoods, compactness and embeddedness give a uniform graph
radius and separation from nonlocal pieces. A sufficiently small
$r_0$ and $c_0$ therefore make the regular comparison ball meet only
its own graph and avoid its endpoints. Integrating the height bound
gives $C\int_{-r}^r|x|\omega(2|x|)\,\dd x\leq Cr^2\omega(2r)$.
At a vertex, sum the two one-sided graph errors against their tangent
rays. Finiteness of the decomposition makes all constants uniform.
\end{proof}

For a polygon or finitely cornered piecewise-$C^{1,\mathrm{Dini}}$ domain $D$ as above
(in particular, for a truncated sector), let
\[
 \mathfrak T(D)=\{\R^2\}
 \cup\{H_a^0:a\in\partial D\setminus\mathcal V(D)\}
 \cup\{C_v^0:v\in\mathcal V(D)\},
\]
where for a polygon $H_a^0=H_e^0$ on the open edge containing $a$, and coincident models are retained
only once.

\begin{definition}\label{def:signed-tangent-model}
The pair $(K,\Phi)$ satisfies \emph{signed tangent-model cancellation} on $D$ if
\begin{equation}\label{eq:signed-model-cancel}
 \int_{G^0\cap\Sph^1}\Phi\bigl(\varsigma\wtK(\omega)\bigr)\,\dd\sigma(\omega)=0
 \quad\text{for every }G^0\in\mathfrak T(D)
 \quad\text{and every }\varsigma\in\{-1,1\}.
\end{equation}
\end{definition}

\subsection{The PDE criterion, the \texorpdfstring{$C^1$}{C1} obstruction, and extension}

For the domains below and $u\in C^\infty(\overline\Omega;\R)$, write
\begin{equation}\label{eq:neumann-budget}
 u_\Omega:=|\Omega|^{-1}\int_\Omega u,\qquad
 N_\Omega(u):=\|\Delta u\|_{L^1(\Omega)}
             +\|\partial_nu\|_{L^1(\partial\Omega)}.
\end{equation}

\begin{theorem}
\label{thm:pde-mazya}
Let $\Omega$ be a bounded finitely cornered
piecewise-$C^{1,\mathrm{Dini}}$ domain, and let $\Phi:\R^2\to\R$ be
positively two-homogeneous and Lipschitz on $\Sph^1$.
Then the following are equivalent:
\begin{enumerate}[label=\textup{(\roman*)},leftmargin=2.2em]
\item for every $G^0\in\mathfrak T(\Omega)$ and
every $\varsigma\in\{-1,1\}$,
\[
 \int_{G^0\cap\Sph^1}\Phi(\varsigma e)\,\dd\sigma(e)=0;
\]
\item there exists $C_{\Omega,\Phi}<\infty$ such that every
$u\in C^\infty(\overline\Omega;\R)$ satisfies
\begin{equation}\label{eq:pde-mazya}
 \left|\int_\Omega\Phi(\nabla u)\,\dd x\right|^{1/2}
 \leq C_{\Omega,\Phi}
 \left(\|\Delta u\|_{L^1(\Omega)}
       +\|\partial_n u\|_{L^1(\partial\Omega)}\right).
\end{equation}
\end{enumerate}
\end{theorem}

\begin{theorem}
\label{cor:convex-c1-pde-failure}
For $K(x)=x/(2\pi|x|^2)$, let
$\Phi:\R^2\to\R$ be nonzero, positively two-homogeneous, and Lipschitz
on $\Sph^1$, with zero integral on every angular half-circle.
There are a bounded convex $C^1$ domain $\Omega=\Omega_\Phi$ with no
vertices, fixed independently of $\eps$, and a unit tangent $\tau$ with
$\Phi(\tau)\ne0$ such that every signed tangent cancellation holds.
For any nonnegative radial
$\rho\in C_c^\infty(B_1)$ with $\int\rho=1$, put
$f_\eps(x)=\eps^{-2}\rho(x/\eps)$. There are $u_\eps\in C_c^\infty(\R^2)$ with
\begin{equation}\label{eq:nondini-pde-budget}
 \|f_\eps\|_1=1,\qquad N_\Omega(u_\eps)\leq2,
 \qquad \nabla u_\eps=K*f_\eps\quad\text{on }\Omega,
\end{equation}
and, as $\eps\downarrow0$,
\begin{equation}\label{eq:nondini-loglog}
 \int_\Omega\Phi(K*f_\eps)
 =\int_\Omega\Phi(\nabla u_\eps)
 =-\frac{2\Phi(\tau)}{(2\pi)^2}\log\log\frac1\eps+O(1).
\end{equation}
Thus both the density inequality and the Newton PDE inequality fail on
$\Omega$, despite cancellation on every tangent model.
\end{theorem}

The next theorem supplies a linear extension with the same data
budget. Its finite-energy assertion is qualitative; the PDE proof uses
the potential representation in \cref{prop:dini-potential-representation}
and the local closure in \cref{lem:local-measure-closure}.

\begin{theorem}
\label{thm:global-extension}
Let $\Omega$ be a bounded finitely cornered
piecewise-$C^{1,\mathrm{Dini}}$ domain. There is a
constant-preserving linear operator
\[
 E_\Omega:C^\infty(\overline\Omega;\R)
 \longrightarrow C(\R^2)\cap W^{1,2}_{\mathrm{loc}}(\R^2)
\]
such that $E_\Omega u=u$ in $\Omega$,
$E_\Omega u-u_\Omega\in W^{1,2}(\R^2)$ is supported in a fixed compact
set, and
\begin{equation}\label{eq:global-extension}
 \|\Delta E_\Omega u\|_{\cM(\R^2)}
 \leq C_\Omega N_\Omega(u).
\end{equation}
All choices are fixed independently of $u$.
The $W^{1,2}$ membership is qualitative.
\end{theorem}

\subsection{Infinite sectors}

For $0<\theta<2\pi$, set
\begin{equation}\label{eq:sector-def}
 W_\theta=\{r e_t:r>0,\ 0<t<\theta\}.
\end{equation}

For $\varrho>0$ we call
\begin{equation}\label{eq:truncated-sector-def}
 \Omega_{\theta,\varrho}=W_\theta\cap B_\varrho
 =\{r e_t:0<r<\varrho,\ 0<t<\theta\}
\end{equation}
a \emph{truncated sector}; rigid images of these domains are included in the same terminology.
Its tangent models are the plane at interior points, tangent half-planes along the open radial sides and
the open circular arc, the model $W_\theta$ at the central point, and quarter-plane cones at the two outer
endpoints. The central point is a nonflat vertex only when $\theta\neq\pi$; for $\theta=\pi$ the two
radial sides form one smooth diameter and the central model is the same half-plane as the noncorner boundary
model.

\begin{definition}
\label{def:generalized-circle-domain}
A \emph{generalized circle} means a Euclidean circle or a line. A
\emph{finite generalized-circle domain} is a bounded connected Lipschitz
domain with finitely many boundary components, each either a Euclidean circle
or a Jordan curve obtained by cyclically concatenating finitely many embedded
open line segments and circular arcs. Distinct arc closures meet only at
prescribed consecutive endpoints (when there are two arcs, they may share
both endpoints), and the two supporting generalized circles at every listed
endpoint have distinct tangent lines. Subdivisions lying on one supporting
generalized circle are merged; holes are allowed.
\end{definition}

Every finite generalized-circle domain is finitely cornered
piecewise-$C^{1,1}$, hence belongs to the Dini class.
Tangential switches between distinct supporting circles are excluded
from this concrete class.

To state the sector condition without an implicit convention, set
\[
 \begin{gathered}
  \Gamma_j=\{r e_j:r>0\},\qquad
  H_j^0=\{z\in\R^2:z\cdot n_j^{\mathrm{in}}>0\},
  \qquad j\in\{0,\theta\},\\
  n_0^{\mathrm{in}}=e_{\pi/2},\qquad
  n_\theta^{\mathrm{in}}=e_{\theta-\pi/2}.
 \end{gathered}
\]
Thus $H_j^0$ is the origin-based inward tangent half-plane along the open ray $\Gamma_j$, and locally
at a noncentral point $a\in\Gamma_j$ the sector agrees with $a+H_j^0$.

\begin{definition}\label{def:tc-sector}
Let $\mathfrak T_\theta$ be the set of distinct elements of
\[
 \{\R^2,H_0^0,H_\theta^0,W_\theta\}.
\]
The pair $(K,\Phi)$ satisfies $(\mathrm{TC}_\theta)$ if
\[
 \int_{G^0\cap\Sph^1}\Phi\bigl(\varsigma\wtK(\omega)\bigr)\,\dd\sigma(\omega)=0
 \qquad(G^0\in\mathfrak T_\theta,\ \varsigma\in\{-1,1\}).
\]
When $\theta=\pi$, one has $H_0^0=H_\pi^0=W_\pi$, so this common
half-plane condition is imposed only once.
\end{definition}


\subsection{The density criterion and its special cases}

We now state the potential theorem used in the PDE criterion,
followed by its polygonal and infinite-sector forms.

\begin{corollary}\label{thm:curved-density}
Assume \eqref{eq:standing-analytic-hypothesis}, and let
$\Omega$ be a finitely cornered piecewise-$C^{1,\mathrm{Dini}}$ domain. Consider the
following two conditions:
\begin{enumerate}[label=\textup{(\roman*)},leftmargin=2.2em]
\item \eqref{eq:signed-model-cancel} holds on every distinct interior,
noncorner boundary, and vertex tangent model and for both signs
$\varsigma\in\{-1,1\}$;
\item there exists $C_{\Omega,K,\Phi}<\infty$ such that, for every
$1<s\leq\infty$ and
$f\in L^1(\R^2;\R)\cap L^s_{\mathrm{loc}}(\R^2;\R)$,
\begin{equation}\label{eq:curved-density}
 \left|\int_\Omega\Phi(K*f(x))\,\dd x\right|
 \leq C_{\Omega,K,\Phi}\|f\|_{L^1(\R^2)}^p.
\end{equation}
\end{enumerate}
Then conditions \textup{(i)} and \textup{(ii)} are equivalent throughout
$0<\alpha<2$. In particular, the equivalence holds for every
finitely cornered piecewise-$C^{1,\beta}$ domain, $0<\beta\leq1$,
by the specialization of \eqref{eq:dini-tangent-modulus} above.
For $1\leq\alpha<2$, the angular kernel may have any H\"older exponent
$0<\gamma\leq1$; for $0<\alpha<1$ the angular Lipschitz assumption is retained.
\end{corollary}

\begin{corollary}\label{thm:polygon-density}
Let $P$ be a bounded connected Lipschitz polygon with finitely many
boundary components and nonflat vertices. Under
\eqref{eq:standing-analytic-hypothesis}, signed cancellation
\eqref{eq:signed-model-cancel} on $P$ is equivalent to
\begin{equation}\label{eq:polygon-density}
 \left|\int_P\Phi(K*f)\,\dd x\right|
 \leq C_{P,K,\Phi}\|f\|_1^p
 \quad\bigl(f\in L^1(\R^2;\R)\cap L^s_{\rm loc}(\R^2;\R),\
             1<s\leq\infty\bigr).
\end{equation}
The constant is independent of $s$ and $f$.
\end{corollary}

\begin{theorem}\label{thm:sector-density}
Assume \eqref{eq:standing-analytic-hypothesis}, and let $0<\theta<2\pi$.
Consider the following two conditions:
\begin{enumerate}[label=\textup{(\roman*)},leftmargin=2.2em]
\item $(\mathrm{TC}_\theta)$ holds;
\item there exists $C_{\theta,K,\Phi}<\infty$ such that, for every
$1<s\leq\infty$ and $f\in L_c^s(\R^2;\R)$ satisfying
\begin{equation}\label{eq:zero-mean}
 \int_{\R^2}f=0,
\end{equation}
one has $\Phi(K*f)\in L^1(W_\theta)$ and
\begin{equation}\label{eq:sector-density}
 \left|\int_{W_\theta}\Phi(K*f(x))\,\dd x\right|
 \leq C_{\theta,K,\Phi}\|f\|_{L^1(\R^2)}^p.
\end{equation}
\end{enumerate}
Then conditions \textup{(i)} and \textup{(ii)} are equivalent throughout
$0<\alpha<2$.

Moreover,  for the implication
\textup{(ii)}$\Rightarrow$\textup{(i)}, it is enough to require
\textup{(ii)} only for zero-mean $f\in C_c^\infty(\R^2;\R)$. The reverse
implication holds for all data in the class stated in \textup{(ii)}.
\end{theorem}

In these three density statements, the estimate constant is independent
of $s$ and the datum; in the sector case it is also independent of the
location and radius of the compact support.
The Newton kernel corresponds to $\alpha=1$ and $p=2$.

\begin{corollary}
\label{cor:generalized-circle-extension-pde}
Every finite generalized-circle domain admits the constant-preserving
Laplacian-measure extension of \cref{thm:global-extension} and satisfies
the necessary-and-sufficient Newton PDE criterion of \cref{thm:pde-mazya}.

\end{corollary}

\begin{proof}
Each line or circular boundary branch has Lipschitz tangent up to its
endpoints, so $\Omega$ is in the Dini class. Apply
\cref{thm:global-extension,thm:pde-mazya}.
\end{proof}

The density sufficiency proofs first treat $f\in L_c^s(\R^2;\R)$,
$1<s\leq\infty$; on infinite sectors impose \eqref{eq:zero-mean}.
On bounded domains, truncation of the source support then gives the
full stated data class.
The annular estimates themselves require only compactly supported $L^1$ data.

\section{Atomic localization and cumulative criteria}
\label{sec:atomic-analysis}
\subsection{Annular decomposition and mixed increments}
Set $K_n(0)=0$ and define the sharp annular truncations by
\begin{equation}\label{eq:annular-kernels}
 K_n=K\one_{\{2^{-n-1}\leq|x|<2^{-n}\}},\qquad
 K_{\leq n}=\sum_{j\leq n}K_j
 =K\one_{\{|x|\geq2^{-n-1}\}}\quad\text{a.e.}
\end{equation}
Their scaling is
\begin{equation}\label{eq:annular-scaling}
 K_n(x)=2^{(2-\alpha)n}K_0(2^nx).
\end{equation}
For $s,t\geq0$, put
\begin{equation}\label{eq:Mp}
 \Mp(s,t)=
 \begin{cases}
  \min\{s^{p-1}t,st^{p-1}\},&1<p\leq2,\\
  \tfrac12(s^{p-1}t+st^{p-1}),&p>2.
 \end{cases}
\end{equation}

\begin{lemma}\label{lem:annular-summation}
Let $D\subset\R^2$ be measurable and let
$U_n,V_{n+1},U\in L^p(D;\R^m)$ satisfy
$U_{n+1}=U_n+V_{n+1}$ and
$U_n\to0$ as $n\to-\infty$, $U_n\to U$ as $n\to+\infty$,
strongly in $L^p(D)$. Then
\begin{align}
 \left|\int_D\Phi(U)\right|
 &\leq\sum_{n\in\mathbb Z}\left|\int_D\Phi(V_{n+1})\right|
 +C_\Phi\sum_{n\in\mathbb Z}
   \int_D\Mp(|U_n|,|V_{n+1}|).
 \label{eq:annular-summation}
\end{align}
An infinite right-hand side is allowed.
\end{lemma}

\begin{proof}
The growth bound and \eqref{eq:phi-lip} imply
\begin{equation}\label{eq:pointwise-increment}
 |\Phi(a+b)-\Phi(a)-\Phi(b)|\leq C_\Phi\Mp(|a|,|b|).
\end{equation}
Indeed, for $|b|\leq|a|$ the left-hand side is at most
$C|a|^{p-1}|b|$; interchange $a,b$ otherwise.
Integrate this inequality and telescope on a finite scale interval;
H\"older's inequality makes every mixed integral finite. The limits
follow from
\begin{equation}\label{eq:phi-lp-continuity}
 \|\Phi(v)-\Phi(w)\|_{L^1(D)}
 \leq C_\Phi\|v-w\|_{L^p(D)}
 (\|v\|_{L^p(D)}+\|w\|_{L^p(D)})^{p-1}.
\end{equation}
\end{proof}

\begin{theorem}\label{thm:stolyarov-mixed}
Under \eqref{eq:standing-analytic-hypothesis}, every real $f\in L^1(\R^2)$ satisfies
\begin{equation}\label{eq:high-mixed}
 \sum_{n\in\mathbb Z}\int_{\R^2}
 \Mp(|K_{\leq n}*f|,|K_{n+1}*f|)
 \leq C_{K,p}\|f\|_1^p.
\end{equation}
The same constant controls every truncated sum $n\geq n_0$,
independently of $n_0\in\mathbb Z$.
\end{theorem}

\begin{proof}
For $p\geq2$, \cref{thm:superquadratic-mixed-direct} proves the bound
with $C_{K,p}=C_p\|\wtK\|_\infty^p$, using only angular boundedness.
For $1<p<2$, \cite[Theorem~2.3, equation~(2.20)]{StolyarovWhole} gives
\begin{equation}\label{eq:threshold-mixed}
 \sum_{j\geq0}\int_{\R^2}
 \Mp(|K_{\leq j}*h|,|K_{j+1}*h|)
 \leq C_{K,p}\|h\|_1^p,\qquad h\in C_c^\infty(\R^2;\R).
\end{equation}
This statement imposes neither zero mean nor a support normalization;
its surrounding cancellation hypotheses can be met by $\Phi\equiv0$.
For smooth $f$ and $k\in\mathbb Z$, apply it to
$f^{[k]}(x)=2^{-2k}f(2^{-k}x)$. Since
\[
 (K_{\leq j}*f^{[k]})(2^kx)
 =2^{k(\alpha-2)}(K_{\leq j+k}*f)(x)
\]
and likewise for $K_{j+1}$, critical homogeneity gives
\eqref{eq:high-mixed} first on $n\geq k$, then on all scales by
monotone convergence as $k\to-\infty$.
For general $f\in L^1$, smooth $L^1$ approximation gives uniform
convergence of both truncated fields at each fixed scale, because
both kernels are bounded. Fatou's lemma on $\mathbb Z\times\R^2$
finishes the proof.
\end{proof}

The direct proof for $p\geq2$ simplifies the positive-kernel argument
outlined in \cite[Section~5, after equation~(5.4)]{StolyarovWhole};
it does not claim a new exponent range. For the domain formulation
see also \cite[equations~(3.7)--(3.9)]{StolyarovDomains}.

\begin{lemma}\label{lem:kernel-endpoints}
Let $f\in L_c^s(\R^2;\R)$, $1<s\leq\infty$, and $T_n=K_{\leq n}*f$.
If $D$ is bounded and measurable, or if $D=\R^2$ and $\int f=0$,
then $K*f,T_n\in L^p(D)$ and, strongly in $L^p(D)$,
\[
 T_n\to0\quad(n\to-\infty),\qquad
 T_n\to K*f\quad(n\to+\infty).
\]
In particular, $\Phi(K*f)\in L^1(D)$.
\end{lemma}

\begin{proof}
Put $\sigma=\min\{s,p\}>1$ and
$1/r=1+1/p-1/\sigma$, so $1\leq r<p$ and $f\in L^\sigma$.
For $K_{>n}=K\one_{\{0<|x|<2^{-n-1}\}}$, local Fubini and Young's inequality give
\begin{equation}\label{eq:density-high-tail}
 \|K*f-T_n\|_p
 \leq\|K_{>n}\|_r\|f\|_\sigma
 \lesssim_{K,r}2^{-2n(1-1/\sigma)}\|f\|_\sigma\to0
 \quad(n\to+\infty).
\end{equation}
On bounded $D$, $T_n$ is bounded and
$\|T_n\|_{L^p(D)}\lesssim|D|^{1/p}2^{(2-\alpha)n}\|f\|_1\to0$
as $n\to-\infty$. These bounds also prove $K*f\in L^p(D)$.

For the full plane, choose $S>0$ with $\supp f\subset B_S$.
Zero mean and angular H\"older regularity give
\begin{equation}\label{eq:far-field}
 |K*f(x)|\lesssim |x|^{\alpha-2-\gamma}
 \int|y|^\gamma|f(y)|\,\dd y,\qquad |x|>2S.
\end{equation}
If $R=2^{-N-1}>4S$, then $T_N$ vanishes on $B_{R-S}$, is bounded by
$C\|f\|_1R^{\alpha-2}$ on the shell $R-S\leq|x|\leq R+S$ of area
$4\pi RS$, and equals $K*f$ outside that shell. Hence
\begin{equation}\label{eq:sector-low-tail}
 \|T_N\|_p^p\lesssim\|f\|_1^p
 \left[\frac SR+\left(\frac SR\right)^{p\gamma}\right]\to0
 \quad(N\to-\infty).
\end{equation}
Together with \eqref{eq:density-high-tail}, this proves all assertions;
$|\Phi(z)|\lesssim|z|^p$ gives integrability.
These endpoint limits are qualitative: their $s$-dependent bounds do
not enter the density estimate constant.
\end{proof}

Thus bounded-domain sufficiency reduces to the annular estimate
\[
 \sum_{n\in\mathbb Z}\left|\int_\Omega\Phi(K_n*f)\right|
 \lesssim\|f\|_1^p,\qquad f\in L_c^1(\R^2;\R),
\]
proved below by the atomic criterion and verified geometrically in
\cref{sec:tangent-geometry}. The mixed bound and endpoint limits recover the
ordinary integral for $L_c^s$ data, $s>1$.

\paragraph{\bf The $L^1$ endpoint.}
Ordinary Lebesgue integrability can fail at $s=1$. For the Newton kernel,
$\Omega=B_1$, and $\Phi(z)=z_1^2-z_2^2$, take
\[
 f(x)=\frac{\one_{\{0<|x|<1/2\}}}
 {|x|^2[\log(e/|x|)]^{3/2}}\in L_c^1.
\]
All signed tangent moments vanish, but radial symmetry gives
\[
 (K*f)(re_t)=\frac{2e_t}{r\sqrt{\log(e/r)}},\qquad
 \int_{B_{1/2}}|\Phi(K*f)|
 =16\int_0^{1/2}\frac{\dd r}{r\log(e/r)}=\infty.
\]
Both signs have infinite integral. Centered-annulus cancellation does
not define a Lebesgue integral; the intermediate $L_c^1$ estimates give
only the specified dyadic limit described in
\cref{sec:tangent-geometry}.

\subsection{Freezing an annular source}
The sharp annular kernels satisfy
\begin{equation}\label{eq:l1-translation}
 \|K_n(\cdot-h)-K_n\|_1
 \lesssim2^{-\alpha n}\min\{1,(2^n|h|)^\gamma\}.
\end{equation}
At scale zero, crossing either boundary circle costs $O(|h|)$;
on the common annulus, angular H\"older regularity gives
$O(|h|^\gamma)$. Large shifts use $2\|K_0\|_1$; dilation gives
\eqref{eq:l1-translation}. Retain the two-sign annular moment defined in
\eqref{eq:atomic-annular-defect}.

\begin{lemma}\label{lem:local-model-defect}
Let $D$ be measurable and let $g\in L^1(\R^2;\R)$ vanish outside a
square $R$. With $M=\int|g|$, every $a\in\R^2$ and $n\in\mathbb Z$
satisfy
\begin{equation}\label{eq:local-model-defect}
 \left|\int_D\Phi(K_n*g)\right|
 \leq C2^{\gamma n}M^{p-1}\int_R|y-a|^\gamma|g(y)|\,\dd y
       +M^p\mathfrak a_{n,D}(a).
\end{equation}
The constant depends only on the analytic data; if $M=0$, both sides vanish.
\end{lemma}

\begin{proof}
Put $s=\int g$, $A=K_n*g$, and $B=sK_n(\cdot-a)$. By
\eqref{eq:phi-lip}, the annular size bound, and
\eqref{eq:l1-translation},
\[
 \|\Phi(A)-\Phi(B)\|_1
 \lesssim 2^{((2-\alpha)(p-1)-\alpha)n}M^{p-1}
 \int_R\min\{1,(2^n|y-a|)^\gamma\}|g(y)|\,\dd y.
\]
Critical homogeneity makes the prefactor $1$. Dropping the minimum
gives the first term in \eqref{eq:local-model-defect}; positive
homogeneity and $|s|\leq M$ give the second, including $s=0$.
\end{proof}

If a dilation-invariant model $G^0$ has signed cancellation, polar
integration makes its signed frozen-atom integral zero. Consequently
\begin{equation}\label{eq:atomic-geometric-bound}
 \mathfrak a_{n,D}(a)
 \lesssim 2^{2n}|(D\triangle(a+G^0))\cap(a+\supp K_n)|.
\end{equation}
This is the only geometric input to the annular summation.

\subsection{Exact finite-range localization}
Fix
\[
 c=\frac{16}{3},\qquad L_n=c2^{-n}.
\]
The factor $16/3$ makes each three-strip assignment square dyadic;
the nine lattices below encode its coordinate residues modulo three.
For $\omega\in\{0,1,2\}^2$, let
\begin{equation}\label{eq:explicit-shifted-grid}
 \cD^\omega
 :=\left\{
  2^{-k}\bigl([0,1)^2+z+(-1)^k\omega/3\bigr):
  k\in\mathbb Z,\ z\in\mathbb Z^2
 \right\}.
\end{equation}
Each $\cD^\omega$ is a nested half-open dyadic lattice: in units of the
finer side length, the change of shift between consecutive generations is
the integer vector $(-1)^k\omega$.

The strip identity and indicator-source decomposition below are those of
\cite[proof of Theorem~4.1, equations~(4.11)--(4.13)]{StolyarovDomains}.
We verify the assignment multiplicity over all integer scales.

\begin{lemma}\label{lem:two-d-local}
\label{lem:assignment-multiplicity}
For every $f\in L_c^1(\R^2;\R)$, there exists a sequence
$(\cQ_n)_{n\in\mathbb Z}$ of finite index sets. Each index
$Q\in\cQ_n$ is associated with a sign $\epsilon_Q$, a source set $E_Q$, a function $g_Q$,
a lattice index $\omega_Q$, and a square $R_Q$. These data satisfy the
following properties, with \textup{(i)--(iv)} holding for every
$n\in\mathbb Z$:
\begin{enumerate}[label=\textup{(\roman*)},leftmargin=2.2em]
\item $\epsilon_Q\in\{-1,1\}$,
      $g_Q=f\one_{E_Q}$, and
      $E_Q\subset R_Q\in\cD^{\omega_Q}$;
\item $\ell(R_Q)=16\,2^{-n}$;
\item pointwise almost everywhere, and hence after integration over every
      measurable set $E$,
\begin{equation}\label{eq:two-d-local}
 \Phi(K_n*f)=\sum_{Q\in\cQ_n}\epsilon_Q\Phi(K_n*g_Q);
\end{equation}
\item $\sum_{Q\in\cQ_n}\one_{E_Q}\leq9$, so
      $\sum_{Q\in\cQ_n}\|g_Q\|_1\leq9\|f\|_1$;
\item over all $n\in\mathbb Z$, at most four pairs $(n,Q)$ with
      $Q\in\cQ_n$ are assigned to each pair $(\omega,R)$.
\end{enumerate}
At a fixed scale, the dependence of the attached data on $n$ is suppressed
in the notation. Across scales, multiplicities are counted with respect
to $(n,Q)$, even when the associated sets coincide.
\end{lemma}

\begin{proof}
Each $K_n*f$ is bounded and compactly supported. Since
$\supp K_n\subset\overline{B(0,2^{-n})}$ and
$2^{1-n}<L_n$, at most two adjacent strips contribute at any point
in either coordinate. Set
\[
 I_{n,j}=[jL_n,(j+1)L_n),\qquad
 I_{n,j}^{0}=I_{n,j},\qquad I_{n,j}^{1}=I_{n,j}\cup I_{n,j+1}.
\]
For $\mathbf j=(j_1,j_2)\in\mathbb Z^2$ and
$\eta=(\eta_1,\eta_2)\in\{0,1\}^2$, put
\[
 E_{\mathbf j,\eta}=I_{n,j_1}^{\eta_1}\times I_{n,j_2}^{\eta_2},
 \qquad g_{\mathbf j,\eta}=f\one_{E_{\mathbf j,\eta}},
 \qquad \epsilon_{\mathbf j,\eta}=(-1)^{\eta_1+\eta_2}.
\]
Apply the strip identity
$\Phi(\sum_j A_j)=\sum_j\Phi(A_j+A_{j+1})-\sum_j\Phi(A_j)$
in each coordinate, with $A_j=K_n*(f\one_{I_{n,j}\times\R})$
for the first coordinate. With one nonzero term the right side is
$2\Phi(A_j)-\Phi(A_j)=\Phi(A_j)$; with two adjacent terms it is
$\Phi(A_j+A_{j+1})$. Since $\Phi(0)=0$, this yields
\eqref{eq:two-d-local}.
Each point belongs to one single and two double strips in each
coordinate, giving overlap at most $3^2=9$.

Assign these sources to
\begin{equation}\label{eq:explicit-target-triple}
 R_{n,\mathbf j}
 =\prod_{i=1}^2[(j_i-1)L_n,(j_i+2)L_n),
 \qquad \ell(R_{n,\mathbf j})=16\,2^{-n}.
\end{equation}
Every $E_{\mathbf j,\eta}$ lies in this square. Put $k=n-4$ and choose
the unique $\omega_i\in\{0,1,2\}$ with
$j_i-1\equiv(-1)^k\omega_i\pmod3$.
The lower-left corner, in units of the side length, is
$((j_1-1)/3,(j_2-1)/3)$, so
$R_{n,\mathbf j}\in\cD^\omega$ by \eqref{eq:explicit-shifted-grid}.
For a fixed pair $(\omega,R)$, its side determines $n$ and its corner
determines $\mathbf j$; only the four choices of $\eta$ remain.
This proves the assignment bound over all integer scales.
Compact support of $f$ makes each index set finite after zero sources
are discarded.
\end{proof}

Thus throughout the sequel
\begin{equation}\label{eq:assigned-scale-comparison}
 c_*=C_*=16,\qquad
 c_*2^{-n}\leq\ell(R_Q)\leq C_*2^{-n},
 \qquad C_{\rm asg}=4.
\end{equation}
We put $\rho_K=1$; by definition,
$\supp K_n\subset\overline{B(0,\rho_K2^{-n})}$.
For a Euclidean square $R$ and $A>0$, write
\begin{equation}\label{eq:cube-enlargement}
 R^{[A]}=\{x\in\R^2:\dist(x,R)<A\ell(R)\}.
\end{equation}
All grid squares are half-open. A function supported in $R$ vanishes
almost everywhere outside $R$; its distributional support lies in
$\overline R$, and $\dist(\cdot,R)=\dist(\cdot,\overline R)$.

\subsection{Fixed-set moments and two-level summation}
The boundary-moment and dyadic-energy estimates in
\cite[Lemmas~4.3 and~4.4]{StolyarovDomains} lead to the following
fixed-set estimate. It sums source moments; source overlap at each scale will
sum the atomic annular defects defined in \eqref{eq:atomic-annular-defect}.

\begin{theorem}
\label{thm:closed-anchor-forest}
Let $p>1$, $0<\gamma\leq1$, let $\mu$ be a finite positive Borel
measure on $\R^2$, let $\cD$ be a half-open dyadic lattice, and let
$F\subset\R^2$ be nonempty and closed. For $A\geq0$ and every subfamily
\[
 \mathscr A\subset\{Q\in\cD:\dist(Q,F)\leq A\ell(Q)\},
\]
one has
\begin{equation}\label{eq:closed-anchor-forest}
 \sum_{Q\in\mathscr A}\frac{\mu(Q)^{p-1}}{\ell(Q)^\gamma}
 \inf_{a\in F}\int_Q|x-a|^\gamma\,\dd\mu(x)
 \leq C_{p,A,\gamma}^{\rm anc}\mu(\R^2)^p,
\end{equation}
where
\[
 C_{p,A,\gamma}^{\rm anc}
 =\frac{(A+\sqrt2)^\gamma+
       (\sqrt2)^\gamma\max\{1,(p-1)^{-1}\}}{1-2^{-\gamma}}.
\]
For $F=\{v\}$, the improved bound is
\begin{equation}\label{eq:fixed-vertex-sum}
 \sum_{Q\in\mathscr A}\frac{\mu(Q)^{p-1}}{\ell(Q)^\gamma}
 \int_Q|x-v|^\gamma\,\dd\mu(x)
 \leq\frac{(A+\sqrt2)^\gamma}{1-2^{-\gamma}}\mu(\R^2)^p.
\end{equation}
The constants are independent of $\mu$, $F$, the subfamily, the number
of generations, and the dyadic shift. Zero-mass cubes contribute zero.
\end{theorem}

\begin{proof}
Put $M=\mu(\R^2)$, $m_Q=\mu(Q)$, $d_F=\dist(\cdot,F)$,
$B=A+\sqrt2$, and $c_\gamma=(1-2^{-\gamma})^{-1}$.
Assume $M>0$, and omit zero-mass cubes whenever a negative power occurs.
Comparison with a nearest point of $F$ to each $y\in Q$, followed by
averaging the resulting scalar inequality in $y$, gives
\[
 \begin{aligned}
 J_Q&:=\inf_{a\in F}\int_Q|x-a|^\gamma\,\dd\mu(x)
       \leq m_Q^{-1}I_Q+\int_Qd_F(y)^\gamma\,\dd\mu(y),\\
 I_Q&:=\iint_{Q\times Q}|x-y|^\gamma\,\dd\mu(x)\dd\mu(y).
 \end{aligned}
\]
No measurable nearest-point selection is needed. If $y\in Q\in\mathscr A$,
then $d_F(y)\leq B\ell(Q)$. For $d_F(y)>0$, only scales
$\ell(Q)\geq d_F(y)/B$ occur, with at most one half-open cube containing
$y$ at each integer generation. Thus
\[
 \sum_{\substack{Q\in\mathscr A\,;\,y\in Q}}
       \frac{d_F(y)^\gamma}{\ell(Q)^\gamma}
 \leq c_\gamma B^\gamma,
\]
while all terms vanish if $d_F(y)=0$. Since $m_Q\leq M$, Tonelli bounds
the distance contribution by $c_\gamma B^\gamma M^p$.
For $F=\{v\}$ this is the whole moment, proving
\eqref{eq:fixed-vertex-sum}.

For $r>1$, set
\[
 \Delta_r(S)=m_S^r-\sum_{C\in\operatorname{ch}_{\cD}(S)}m_C^r\geq0.
\]
The half-open children partition their parent. Telescoping over a finite
interval of integer generations and then using monotone convergence gives
$\sum_{S\in\cD}\Delta_r(S)\leq M^r$.
The defect $\Delta_2(S)$ is the ordered-pair mass in distinct children.
Every off-diagonal pair in $Q$ has a unique first child split, and the
diagonal contributes zero. Consequently
\[
 I_Q\leq(\sqrt2)^\gamma\sum_{S\subseteq Q}\ell(S)^\gamma\Delta_2(S).
\]
Put $q=\min\{p,2\}$. For positive-mass $S\subseteq Q$,
$m_Q^{p-2}\leq M^{p-q}m_S^{q-2}$. The ancestor weights sum to
$c_\gamma$. Writing $t_C=m_C/m_S$, the inequality
$1-t^{q-1}\geq(q-1)(1-t)$ for $0\leq t\leq1$ gives
\[
 m_S^{q-2}\Delta_2(S)\leq(q-1)^{-1}\Delta_q(S).
\]
Tonelli therefore bounds the pair contribution by
\[
 \sum_{\substack{Q\in\mathscr A\,;\,m_Q>0}}
       \frac{m_Q^{p-2}I_Q}{\ell(Q)^\gamma}
 \leq(\sqrt2)^\gamma c_\gamma M^{p-q}
       \sum_{\substack{S\in\cD\,;\,m_S>0}}m_S^{q-2}\Delta_2(S)
 \leq\frac{(\sqrt2)^\gamma c_\gamma}{q-1}M^p.
\]
Adding the distance contribution proves \eqref{eq:closed-anchor-forest},
since $(q-1)^{-1}=\max\{1,(p-1)^{-1}\}$.
\end{proof}

\begin{remark}
The anchor set must remain fixed. For $Q_k=[0,2^{-k})^2$,
$\mu=\delta_0$, and $\mathcal A_{Q_k}=\{(2^{-k},2^{-k})\}$,
one has $\dist(Q_k,\mathcal A_{Q_k})\lesssim\ell(Q_k)$, but every
normalized $\gamma$-moment equals $(\sqrt2)^\gamma$ and their sum diverges.
\end{remark}

Projection onto $\overline R$ and density of a half-open square $R$ in
its closure give, for every finite positive measure $\mu$,
\begin{equation}\label{eq:interior-projection-identity}
 \inf_{a\in\R^2}\int_R|x-a|^\gamma\,\dd\mu(x)
 =\inf_{a\in\overline R}\int_R|x-a|^\gamma\,\dd\mu(x)
 =\inf_{a\in R}\int_R|x-a|^\gamma\,\dd\mu(x).
\end{equation}

\begin{proposition}
\label{prop:two-level-annular}
Assume \eqref{eq:standing-analytic-hypothesis} and signed plane
cancellation. Let $D\subset\R^2$ be open with nonempty boundary $F$,
let $V\subset F$ be closed (possibly empty), and let $c>0$ and
$J\subset\mathbb Z$. Put $r_n=\rho_K2^{-n}$ and
$\dist(a,\varnothing)=+\infty$.
Suppose $\epsilon_n\geq0$ and $\sum_{n\in J}\epsilon_n<\infty$, and that
\begin{equation}\label{eq:two-level-atomic-input}
 \mathfrak a_{n,D}(a)\leq\epsilon_n
 \quad\begin{cases}
 a\in F\setminus V,\quad r_n<c\dist(a,V),\\
 a\in V,
 \end{cases}
 \qquad n\in J.
\end{equation}
Then every $f\in L_c^1(\R^2;\R)$ satisfies
\begin{equation}\label{eq:two-level-annular}
 \sum_{n\in J}\left|\int_D\Phi(K_n*f)\,\dd x\right|
 \leq C\left(1+\sum_{n\in J}\epsilon_n\right)\|f\|_1^p,
\end{equation}
where $C$ depends only on the analytic data and $c$.
\end{proposition}

\begin{proof}
Use the sources $g_Q$ and squares $R_Q$ from \cref{lem:two-d-local},
writing $M_Q=\|g_Q\|_1$ and
$J_g(z)=\int_{R_Q}|y-z|^\gamma|g_Q(y)|\,\dd y$.
Set $\kappa=1+\rho_K/c_*$ and put
\[
 H=\kappa+\sqrt2,\qquad A_V=H+\frac{\rho_K}{cc_*}.
\]
The convolution support lies in $R_Q^{[\kappa]}$.
Omit zero sources and squares with $R_Q^{[\kappa]}\cap D=\varnothing$.
For an interior square $R_Q^{[\kappa]}\subset D$, choose a moment
minimizer $a_Q\in\overline{R_Q}$ using
\eqref{eq:interior-projection-identity}. Its fixed anchor set is
$\R^2$, and $\mathfrak a_{n,D}(a_Q)=0$ by plane cancellation.

For every remaining square $R=R_Q$, convexity gives a point
$q\in F\cap R^{[\kappa]}$. The continuous coercive function $J_g$ attains its minimum at some $a\in F$,
even when $F$ is unbounded. Comparison with $q$ gives
\[
 M_Q\dist(a,R)^\gamma\leq J_g(a)\leq J_g(q)
 \leq M_QH^\gamma\ell(R)^\gamma,
 \qquad\dist(R,F)\leq\kappa\ell(R).
\]
If $r_n<c\dist(a,V)$, retain $a_Q=a$ and the fixed set $F$.
Otherwise $V\ne\varnothing$ and a nearest $v\in V$ gives
\[
 \dist(R,V)\leq\dist(a,R)+\dist(a,V)\leq A_V\ell(R).
\]
Choose instead $a_Q\in\arg\min_VJ_g$, which exists by coercivity.
Its fixed anchor set is $V$; it need not be the nearest point $v$.
In either case \eqref{eq:two-level-atomic-input} bounds the atomic
defect by $\epsilon_n$. Thus the sources have three fixed anchor sets
$\R^2,F,V$, with proximity constants $0,\kappa,A_V$; the last is
omitted when $V=\varnothing$.

Put $M=\|f\|_1$ and $\dd\mu=|f|\,\dd x$. Localization and
\cref{lem:local-model-defect} bound the left side of
\eqref{eq:two-level-annular} by
\[
 C\sum_{n\in J}\sum_Q
 \left(2^{\gamma n}M_Q^{p-1}J_g(a_Q)+\epsilon_nM_Q^p\right).
\]
For each anchor class and each of the nine lattices, group by the
assigned square $R$. At most four indices have the same assignment;
$|g_Q|\leq|f|\one_R$ and
$2^{\gamma n}=16^\gamma\ell(R)^{-\gamma}$. Hence
\cref{thm:closed-anchor-forest} bounds the entire moment sum by $CM^p$.
The per-scale overlap gives
\[
 \sum_{n\in J}\epsilon_n\sum_QM_Q^p
 \leq M^{p-1}\sum_{n\in J}\epsilon_n\sum_QM_Q
 \leq9M^p\sum_{n\in J}\epsilon_n.
\]
This proves the assertion. Only nonnegative majorants were summed,
so arbitrary integer-scale sets $J$ are allowed.
\end{proof}

\subsection{A cumulative atomic criterion on arbitrary bounded open sets}

We prove the principal atomic criterion stated in the introduction.

\begin{proof}[Proof of \cref{thm:atomic-criterion}]
Write $F=\partial\Omega$ and
$H_n^\varsigma(x)=\Phi(\varsigma K_n(x))$.  The annular size bound,
\eqref{eq:phi-lip}, and \eqref{eq:l1-translation} give
\[
 \|H_n^\varsigma(\cdot-a)-H_n^\varsigma(\cdot-b)\|_1
 \le C\min\{1,(2^n|a-b|)^\gamma\}.
\]
Indeed, the scale factor is
$2^{((2-\alpha)(p-1)-\alpha)n}=1$.  Taking the maximum over the two
signs therefore gives
\begin{equation}\label{eq:atomic-center-translation}
 |\mathfrak a_{n,\Omega}(a)-\mathfrak a_{n,\Omega}(b)|
 \le C2^{\gamma n}|a-b|^\gamma.
\end{equation}

Fix once and for all a Borel map $P:\R^2\to F$ such that
$P(y)=y$ on $F$ and
$|P(y)-y|\le2\dist(y,F)$ elsewhere.  One elementary construction is
to choose a countable dense sequence $(b_j)$ in $F$ and, for $y\notin F$,
take the first $b_j$ with $|y-b_j|<2\dist(y,F)$.
In particular, for every $a\in F$,
\begin{equation}\label{eq:atomic-fixed-projection}
 |P(y)-a|\le3|y-a|.
\end{equation}
The essential point is that $P$ is independent of both the scale and the
localized source.

Let $f\in L_c^1$, $M=\|f\|_1$, and use the sources $g_Q$ and assigned
squares $R_Q$ of \cref{lem:two-d-local}.  Put
$M_Q=\|g_Q\|_1$ and
$J_Q(a)=\int_{R_Q}|y-a|^\gamma|g_Q(y)|\,\dd y$.
Take $\kappa=1+\rho_K/c_*$, so that the convolution support is in
$R_Q^{[\kappa]}$.  Discard zero sources and those for which this
enlargement misses $\Omega$.
If $R_Q^{[\kappa]}\subset\Omega$, choose a moment minimizer
$a_Q\in\overline{R_Q}$ as in
\eqref{eq:interior-projection-identity}.  Its frozen annulus lies in
$\Omega$, so plane cancellation gives
$\mathfrak a_{n,\Omega}(a_Q)=0$.

For every remaining source the convex enlargement meets $F$, hence
$\dist(R_Q,F)\le\kappa\ell(R_Q)$.  Choose instead a minimizer
$a_Q\in F$ of $J_Q$; it exists since $F$ is compact.  Comparison with
a point of $F\cap R_Q^{[\kappa]}$ gives the fixed proximity bounds
used in the closed-set moment theorem.  Grouping assignments by their
nine lattices and using the multiplicity at most four, that theorem yields
\begin{equation}\label{eq:atomic-source-moment-payment}
 \sum_{n,Q}2^{\gamma n}M_Q^{p-1}J_Q(a_Q)
 \le C M^p.
\end{equation}
Here the fixed anchor sets are $\R^2$ for interior sources and $F$ for
the other sources.  Explicitly, $|g_Q|\le|f|\one_{R_Q}$ implies
$M_Q\le\int_{R_Q}|f|$ and
$\inf_{a\in F}J_Q(a)\le
\inf_{a\in F}\int_{R_Q}|y-a|^\gamma|f(y)|\,\dd y$;
the same argument applies with $F=\R^2$.

For a boundary source with $M_Q>0$, use
\eqref{eq:atomic-center-translation} with $b=P(y)$.
Integrate against $M_Q^{p-1}|g_Q(y)|\,\dd y$ and apply
\eqref{eq:atomic-fixed-projection}.  This gives
\[
 \begin{aligned}
 M_Q^p\mathfrak a_{n,\Omega}(a_Q)
 &\le M_Q^{p-1}\int|g_Q(y)|
               \mathfrak a_{n,\Omega}(P(y))\,\dd y\\
 &\quad+C2^{\gamma n}M_Q^{p-1}J_Q(a_Q).
 \end{aligned}
\]
The second terms sum by \eqref{eq:atomic-source-moment-payment}.
For the first terms, $M_Q\le M$ and the overlap bound
$\sum_Q|g_Q|\le9|f|$ at each scale imply
\[
 \begin{aligned}
 &\sum_{n,Q}M_Q^{p-1}\int|g_Q(y)|
                       \mathfrak a_{n,\Omega}(P(y))\,\dd y\\
 &\qquad\le9M^{p-1}\int|f(y)|
                 \sum_n\mathfrak a_{n,\Omega}(P(y))\,\dd y
 \le9\mathcal B_{K,\Phi}(\Omega)M^p.
 \end{aligned}
\]
All sums here have nonnegative terms.  Localization followed by
\cref{lem:local-model-defect} now proves
\eqref{eq:atomic-cumulative-annular}.

For compactly supported $L^s$ data, $s>1$, apply
\cref{lem:kernel-endpoints,lem:annular-summation,thm:stolyarov-mixed}
to obtain \eqref{eq:atomic-cumulative-density} and ordinary Lebesgue
integrability.  For general $f\in L^1\cap L^s_{\rm loc}$, let
$f_j=f\one_{B_j}$ and $R_\Omega=\sup_{x\in\Omega}|x|$.  For
$j>R_\Omega+1$,
\[
 \sup_{x\in\Omega}|K*(f-f_j)(x)|
 \le C_K(j-R_\Omega)^{\alpha-2}\|f-f_j\|_1\longrightarrow0.
\]
The compact-data result and this estimate give $K*f\in L^p(\Omega)$
and strong $L^p(\Omega)$ convergence of $K*f_j$ to $K*f$.
Equation \eqref{eq:phi-lp-continuity} passes the estimate to the limit.
\end{proof}

\subsection{A necessary signed point-source condition}
\begin{proposition}\label{prop:atomic-necessary-primitive}
Assume the standing analytic hypotheses and let $\Omega$ be a nonempty
bounded open set. If the density estimate holds with constant $C_0$
for all smooth compactly supported real data, then
\begin{equation}\label{eq:atomic-necessary-primitive}
 \mathcal P_{K,\Phi}(\Omega):=
 \sup_{a\in\R^2,\ \varepsilon>0,\ \sigma=\pm1}
 \left|\int_{\Omega\setminus B(a,\varepsilon)}
                 \Phi(\sigma K(x-a))\,\dd x\right|
 \le C_0+C_{K,\Phi}.
\end{equation}
No boundary regularity or tangent model is required.
\end{proposition}
\begin{proof}
Fix one nonnegative smooth unit-mass mollifier supported in $B_1$
and write $\rho_{\varepsilon,a}(x)=\varepsilon^{-2}
\rho((x-a)/\varepsilon)$. Test with $\sigma\rho_{\varepsilon,a}$.
On $B(a,2\varepsilon)$,
local integrability and homogeneity give
$|K*\rho_{\varepsilon,a}|\le C\varepsilon^{\alpha-2}$, so the
nonlinear core integral is $O(1)$ by $p(2-\alpha)=2$.
Outside this ball, the angular H\"older bound gives
\[
 |K*\rho_{\varepsilon,a}(x)-K(x-a)|
 \le C\varepsilon^\gamma|x-a|^{\alpha-2-\gamma}.
\]
Equation \eqref{eq:phi-lip} makes the total nonlinear replacement
error at most
$C\varepsilon^\gamma\int_{2\varepsilon}^\infty r^{-1-\gamma}\,\dd r
\le C$. The density test bounds the remaining tail uniformly.
Replacing $2\varepsilon$ by an arbitrary positive truncation radius
proves the assertion for both signs and every centre.
\end{proof}
For general $\Phi$, we use \cref{prop:atomic-necessary-primitive}
here only as a necessary condition. The quadratic converse proved here is
\cref{thm:quad-cumulative-criterion}.

\section{Tangent models and non-Dini geometry}
\label{sec:tangent-geometry}
\subsection{Logarithmic obstructions and a fixed convex counterexample}
\paragraph{The $C^1$ comparison class.}
A \emph{finitely cornered piecewise-$C^1$ domain} is a bounded connected
Lipschitz domain with the same finite boundary decomposition and nonflat
transverse vertices as in \cref{sec:kernels-domains}, but with its regular
boundary pieces only $C^1$ up to their endpoints. Fix a finite regular $C^1$ atlas and a common
bounded nondecreasing unit-tangent modulus $\omega$, with
$\omega(r)>0$ for $r>0$ and $\omega(r)\to0$ as $r\downarrow0$; the Dini
integral is not required. The graph and vertex comparisons in
\cref{lem:piecewise-model-bounds} use only this tangent control and remain
valid with the same $r^2\omega(2r)$ area remainder. Their proof does not
use Dini integrability. The larger class is used below for necessary
conditions only. Let $\gamma\in(0,1]$ denote the angular regularity exponent
of $K$; the Lipschitz case corresponds to $\gamma=1$.

\begin{lemma}\label{lem:tangent-concentration}
Assume \eqref{eq:standing-analytic-hypothesis}. Let $D$ be a bounded
finitely cornered piecewise-$C^1$ domain in the comparison class above,
including a polygon, or an infinite sector $W_\theta$.
Fix an interior point, a noncorner boundary point, or a vertex
$a\in\overline D$, with origin-based tangent model $G_a^0$.
There are a fixed $r_a>0$ and a bounded nondecreasing local modulus
$\eta_a(r)\to0$ as $r\downarrow0$ such that
\[
 |(D\triangle(a+G_a^0))\cap B(a,r)|\leq Cr^2\eta_a(r),
 \qquad 0<r\leq r_a.
\]
At interior points and for polygons or sectors one may take
$\eta_a=0$; otherwise one may take $\eta_a(r)=\omega(2r)$.
Fix $\rho\in C_c^\infty(B_1)$ with $\rho\geq0$ and $\int\rho=1$, and put
\[
 \rho_{\eps,a}(x)=\eps^{-2}\rho\!\left(\frac{x-a}{\eps}\right),
 \qquad f_\eps^\varsigma=\varsigma(\rho_{\eps,a}-g),
\]
where $g=0$ when $D$ is bounded and $g=\rho_{r_a,a}$ when $D$ is a sector.
For both signs and all sufficiently small $\eps>0$,
\begin{equation}\label{eq:tangent-log}
 \int_D\Phi(K*f_\eps^\varsigma)\,\dd x
 =\log\frac1\eps
   \int_{G_a^0\cap\Sph^1}\Phi(\varsigma\wtK)\,\dd\sigma
   +O(R_a(\eps)),
\end{equation}
where
\begin{equation}\label{eq:c1-tangent-remainder}
 R_a(\eps):=1+\int_{2\eps}^{r_a}\frac{\eta_a(r)}r\,\dd r
 =o\!\left(\log\frac1\eps\right).
\end{equation}
If $\eta_a$ is Dini, including the exact-model case $\eta_a=0$,
then $R_a(\eps)=O(1)$. In the sector case
$\Phi(K*f_\eps^\varsigma)\in L^1(D)$.
All constants may depend on the fixed point and data, but not on
$\eps$ or the sign.
\end{lemma}

\begin{proof}
The local geometric bound follows from
\cref{lem:piecewise-model-bounds} and its $C^1$ comparison above;
choose $r_a$ below the relevant model radius. For polygons and
sectors the model is exact on this ball.
On the annulus $2\eps<r=|x-a|<r_a$, homogeneity and angular
$\gamma$-H\"older regularity give
\begin{equation}\label{eq:approx-delta}
 K*\rho_{\eps,a}(x)=K(x-a)
 +O(\eps^\gamma r^{\alpha-2-\gamma}),\qquad
 |K*\rho_{\eps,a}(x)|+|K(x-a)|\lesssim r^{\alpha-2}.
\end{equation}
The assumptions on $\Phi$ imply
\begin{equation}\label{eq:phi-lip}
 |\Phi(z)-\Phi(w)|
 \leq C_\Phi|z-w|(|z|+|w|)^{p-1}.
\end{equation}
Since $p(2-\alpha)=2$, replacing
$\varsigma K*\rho_{\eps,a}$ by $\varsigma K(\,\cdot-a)$
on the annulus costs at most
\[
 C\eps^\gamma\int_{2\eps}^{r_a}r^{-1-\gamma}\,\dd r=O(1).
\]
Writing $E_a=D\triangle(a+G_a^0)$, truncated layer-cake integration
and the local area bound give
\[
 \int_{E_a\cap\{2\eps<|x-a|<r_a\}}|x-a|^{-2}\,\dd x
 \lesssim\eta_a(r_a)+\int_{2\eps}^{r_a}\frac{\eta_a(r)}r\,\dd r
 \lesssim R_a(\eps).
\]
Thus replacing $D$ by its tangent model costs $O(R_a(\eps))$.
Polar integration on that model gives the angular coefficient in
\eqref{eq:tangent-log} times $\log(r_a/(2\eps))$.

The compensator field $K*g$ is bounded near $a$, while
$|K*\rho_{\eps,a}(x)|\lesssim(|x-a|+\eps)^{\alpha-2}$.
By \eqref{eq:phi-lip}, adding $-\varsigma K*g$ costs $O(1)$ on the
fixed local ball, because
\[
 \int_0^{r_a}\bigl(1+(r+\eps)^{-\alpha}\bigr)r\,\dd r\lesssim1,
 \qquad 0<\alpha<2.
\]
Critical scaling bounds the unperturbed nonlinear integral on
$B(a,2\eps)$ by $O(1)$; the fixed field $K*\rho$ is bounded on $B_2$.
On every fixed bounded set outside $B(a,r_a)$ the fields are uniformly
bounded, so this proves the required outer bound when $D$ is bounded.

For a sector, $f_\eps^\varsigma$ has zero mean, uniformly bounded
$L^1$ norm, and support in one fixed ball $B_R$.
The zero-mean identity and the same angular finite-difference bound give
\[
 |K*f_\eps^\varsigma(x)|
 \lesssim |x|^{\alpha-2-\gamma}
       \int |y|^\gamma|f_\eps^\varsigma(y)|\,\dd y
 \lesssim |x|^{\alpha-2-\gamma},\qquad |x|>2R.
\]
Consequently its nonlinear tail is uniformly bounded by a constant
times $\int_{2R}^\infty r^{-1-p\gamma}\,\dd r<\infty$.
Together with the local bounds, this proves absolute integrability
and \eqref{eq:tangent-log} for sectors as well.
Finally, splitting the integral in \eqref{eq:c1-tangent-remainder}
at a fixed small radius shows that $\eta_a(r)\to0$ gives the stated
$o(\log(1/\eps))$ remainder; Dini integrability gives $O(1)$.
All estimates hold uniformly over the two signs.
\end{proof}

For each distinct tangent model of the bounded domain or sector,
choose one representative point $a$ and apply
\cref{lem:tangent-concentration} with both signs.
The source masses are at most $1$ for bounded domains and at most $2$
for sectors; the latter sources also have zero mean.
The assumed density inequality therefore bounds the left-hand side
of \eqref{eq:tangent-log} uniformly in $\eps$.
Divide by $\log(1/\eps)$ and let $\eps\downarrow0$ to force the
corresponding signed angular coefficient to vanish.
This proves necessity in
\cref{thm:polygon-density,thm:curved-density,thm:sector-density},
and the same necessary condition on the piecewise-$C^1$ comparison
class, without asserting sufficiency there. Coincident sector models
are imposed only once, including when $\theta=\pi$.

\begin{proof}[Proof of \cref{cor:convex-c1-pde-failure}]
Choose $\tau$ with $\Phi(\tau)\ne0$ and use $(\tau,\tau^\perp)$ as
coordinates. For $F(t)=\Phi(e_t)$, differentiating
$\int_a^{a+\pi}F=0$ gives $F(t+\pi)=F(t)$; hence $\Phi$ is even and
$F(0)=\Phi(\tau)$.

For $0<r\leq1$, put $L(r)=\log(e/r)$ and
$\psi(r)=r^2+2r/L(r)$, with $\psi(0)=0$, and define
\[
 \Omega=\{x\in B_1:\psi(|x|)-2x_2<0\}.
\]
The function $r/L(r)$ is increasing and convex, with derivative tending
to zero at $0$. Thus $\psi(|x|)-2x_2$ is convex and $C^1$, smooth
away from $0$. It is positive on $\partial B_1$ and negative at
$(0,r)$ for small $r>0$. A critical point of a differentiable convex
function is a minimum, so its zero level is regular. Consequently
$\Omega$ is a bounded convex $C^1$ domain with tangent $\tau$ at $0$.
Its plane and half-plane moments cancel.

For the prescribed radial mollifier let
$v_\eps=\Gamma*\rho_\eps$, $\Gamma(x)=(2\pi)^{-1}\log|x|$.
Radial symmetry gives
\[
 \Delta v_\eps=\rho_\eps,\qquad
 \nabla v_\eps(x)=\frac{q_\eps(|x|)x}{2\pi|x|^2},\qquad
 q_\eps(r)=\int_{B_r}\rho_\eps\in[0,1],
\]
with $q_\eps=1$ for $r\geq\eps$. A fixed cutoff equal to one near
$\overline\Omega$ makes $u_\eps$ ambiently smooth and compactly
supported without changing these identities there. Convexity and
$0\in\partial\Omega$ give $n_x\cdot x\geq0$, so the divergence theorem
implies
$N_\Omega(u_\eps)=2\int_\Omega\rho_\eps\leq2$.

At small radius $r$, the domain occupies
$(\theta(r),\pi-\theta(r))$, where
\[
 \theta(r)=\arcsin\!\left(\frac r2+\frac1{L(r)}\right)
 =\frac1{L(r)}+O\!\left(L(r)^{-3}+r\right).
\]
Since $F$ is Lipschitz, $\pi$-periodic and has zero half-circle mean,
its integral on this interval is
$-2F(0)/L(r)+O(L(r)^{-2}+r)$.
The inner core contributes $O(1)$ by
$\|\nabla v_\eps\|_{L^\infty(B_\eps)}\lesssim_\rho\eps^{-1}$;
the region outside a fixed small $B_{r_0}$ also contributes $O(1)$.
On the intervening annulus $q_\eps=1$, and therefore
\[
 \int_\Omega\Phi(\nabla u_\eps)
 =\frac1{(2\pi)^2}\int_\eps^{r_0}
 \left(-\frac{2F(0)}{L(r)}+O(L(r)^{-2}+r)\right)\frac{\dd r}r+O(1).
\]
This is \eqref{eq:nondini-loglog}, with nonzero coefficient, and proves
both failures.
\end{proof}

Taking $\tau=e_0$, both $\Phi_2(\xi)=\xi_1^2-\xi_2^2$ and
$\Phi_4(re_t)=r^2\cos4t$ give coefficient $-2/(2\pi)^2$.

\begin{remark}
On the positive boundary branch $x(r)=re_{\theta(r)}$, radial flux is
exactly
\[
 |\partial_n v_\eps|\,\dd s
 =\frac{q_\eps(r)}{2\pi}\,\dd\theta(r),\qquad
 \theta'(r)=
 \frac{\frac12+\frac1{rL(r)^2}}
      {\sqrt{1-(\frac r2+L(r)^{-1})^2}}>0.
\]
Thus its $L^1$ normal flux measures finite angular variation, whereas
$\int_0^{r_0}\theta(r)\,\dd r/r=\infty$ accumulates the angular offset.
More generally, replace $\psi(r)$ by $r^2+2rL(r)^{-\beta}$,
$\beta>0$. The same convexity and flux argument applies; the tangent
modulus is comparable to $L(r)^{-\beta}$ and is Dini exactly for
$\beta>1$. In that case the density and PDE inequalities hold by the
two criteria. For $0<\beta<1$,
\[
 \int_{\Omega_\beta}\Phi(\nabla u_\eps)
 =-\frac{2\Phi(\tau)}{(2\pi)^2(1-\beta)}
   \left(\log\frac1\eps\right)^{1-\beta}
   +o\!\left(\left(\log\frac1\eps\right)^{1-\beta}\right).
\]
Indeed, $\theta_\beta(r)=\arcsin(r/2+L(r)^{-\beta})\sim L(r)^{-\beta}$,
and the integrated $O(\theta_\beta^2)$ error is lower order.
The case $\beta=1$ is \eqref{eq:nondini-loglog}; no bounded remainder is
asserted for all $\beta<1$.
\end{remark}

\subsection{Dini verification of the cumulative criterion}
\begin{corollary}
\label{cor:atomic-dini-verification}
Let $\Omega$ be a finitely cornered piecewise-$C^{1,\mathrm{Dini}}$
domain.  If signed cancellation holds on all its tangent models, then
\begin{equation}\label{eq:atomic-dini-budget}
 \mathcal B_{K,\Phi}(\Omega)
 \le C_\Omega\left(1+
            \int_0^{2r_0}\frac{\omega(r)}r\,\dd r\right)<\infty.
\end{equation}
Here $r_0$ is a fixed sufficiently small model radius.  In particular,
the bounded-domain density sufficiency follows from
\cref{thm:atomic-criterion}.
\end{corollary}

\begin{proof}
Choose $n_0$ with $\rho_K2^{-n_0}<r_0$, decreasing $r_0$ if necessary.
The model bounds and \eqref{eq:atomic-geometric-bound} give
$\mathfrak a_{n,\Omega}(v)\le C\omega(2^{1-n})$ at every vertex
$v$, and the same estimate at every nonvertex $a$ whenever
$r_n:=\rho_K2^{-n}<c_0\dist(a,\mathcal V(\Omega))$.
If the vertex set is empty, these estimates cover every fine-scale
boundary atom.  Otherwise fix a nonvertex $a$, put
$d=\dist(a,\mathcal V(\Omega))>0$, and choose a nearest vertex $v$.
On the remaining scales $r_n\ge c_0d$, the sharper minimum in the
translation estimate above gives
\[
 \mathfrak a_{n,\Omega}(a)
 \le C\omega(2^{1-n})+
       C\min\{1,(d/r_n)^\gamma\}.
\]
Uniformly in $d>0$,
$\sum_{r_n\ge c_0d}\min\{1,(d/r_n)^\gamma\}<\infty$,
with a bound depending only on $c_0,\gamma$.  The vertex case was
already included.  Thus the fine-scale sum is bounded uniformly in
$a$ by a constant plus the dyadic Dini sum.  Finally,
$\mathfrak a_{n,\Omega}(a)\le C|\Omega|2^{2n}$ for all $a$, so
the coarse scales $n<n_0$ contribute a uniform geometric sum.
The usual dyadic comparison with the Dini integral proves
\eqref{eq:atomic-dini-budget}.
\end{proof}

\subsection{The Dini density criterion and sharp truncations}
Throughout this subsection, fix a finitely cornered
piecewise-$C^{1,\mathrm{Dini}}$ domain $\Omega$, its common modulus $\omega$,
and the constants $r_0,C_0,c_0$ from \cref{lem:piecewise-model-bounds}.

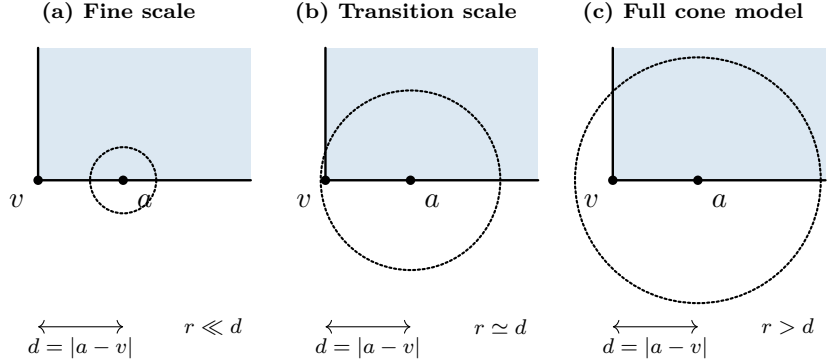
\begin{figure}[!htbp]
\centering
\begin{tikzpicture}[x=1.25cm,y=1.25cm,line cap=round,line join=round]
\foreach \shift/\rad/\tag/\txt in
 {0/.35/{(a) Fine scale}/{$r\ll d$},
  3.8/.95/{(b) Transition scale}/{$r\simeq d$},
  7.6/1.30/{(c) Full cone model}/{$r>d$}} {
 \begin{scope}[xshift=\shift cm]
  \fill[modelblue] (0,0) rectangle (2.25,1.4);
  \draw[geometric boundary] (0,1.4)--(0,0)--(2.25,0);
  \draw[kernel annulus] (.9,0) circle (\rad);
  \node[figure point,label=below left:$v$] at (0,0) {};
  \node[figure point,label=below right:$a$] at (.9,0) {};
  \draw[<->,thin] (0,-1.55)--(.9,-1.55)
     node[midway,below,font=\scriptsize] {$d=|a-v|$};
  \node[font=\scriptsize\bfseries,align=center] at (.85,1.77) {\tag};
  \node[font=\scriptsize] at (1.85,-1.55) {\txt};
 \end{scope}
}
\end{tikzpicture}
\caption[The transition from one side to the vertex cone]{Three scales
near the same vertex. In the quadrant pictured, a ball of radius $r<d$
about the noncorner boundary point $a$ meets only the horizontal side, so the
half-plane model is exact there. For $r>d$, the ball meets the second side and the full tangent cone is needed. The balls
represent the neighborhood sampled by a localized annulus; the proof
uses expanded source cubes with fixed comparison constants.
\Cref{cor:atomic-dini-verification,prop:curved-annular} implement this choice for
curved sides, with the quantified tangent-model error.}
\label{fig:tangent-models}
\end{figure}
\FloatBarrier

\begin{proposition}
\label{prop:curved-annular}
Assume signed cancellation on every tangent model of $\Omega$. Then a
constant $C_{\Omega,K,\Phi}$, independent of $f$ and the dyadic scale,
satisfies
\begin{equation}\label{eq:curved-annular}
 \sum_{n\in\mathbb Z}
 \left|\int_\Omega\Phi(K_n*f)\,\dd x\right|
 \leq C_{\Omega,K,\Phi}\|f\|_1^p,
 \qquad f\in L_c^1(\R^2;\R).
\end{equation}
\end{proposition}

\begin{proof}
Apply \cref{cor:atomic-dini-verification,thm:atomic-criterion}.
\end{proof}

\Cref{cor:atomic-dini-verification,thm:atomic-criterion} give
sufficiency in \cref{thm:curved-density} for the full stated data class;
\cref{lem:tangent-concentration} gives necessity.

Every bounded Lipschitz polygon considered here is a finitely cornered
piecewise-$C^{1,1}$, hence piecewise-$C^{1,\mathrm{Dini}}$, domain, and its
tangent-model family in the two definitions is the same. Therefore
\cref{thm:polygon-density} follows directly from \cref{thm:curved-density}.

\paragraph{\bf Dyadic truncation for integrable data.}
Under the same signed cancellation assumptions, put
$I_n(f)=\int_\Omega\Phi(K_{\leq n}*f)$ for $f\in L_c^1(\R^2;\R)$.
Each integral is finite because the truncated field is bounded.
The increment bound \eqref{eq:pointwise-increment},
\eqref{eq:curved-annular}, and \cref{thm:stolyarov-mixed} give
\[
 \sum_{n\in\mathbb Z}|I_{n+1}(f)-I_n(f)|\leq C\|f\|_1^p.
\]
Since $|I_n(f)|\lesssim|\Omega|2^{2n}\|f\|_1^p\to0$ as
$n\to-\infty$, the limit $\lim_{n\to+\infty}I_n(f)$ exists and has
absolute value at most $C\|f\|_1^p$.
For $f\in L_c^s$, $1<s\leq\infty$, \cref{lem:kernel-endpoints} and
\eqref{eq:phi-lp-continuity} identify it with the ordinary integral.
For general $L_c^1$ data, this conclusion concerns the specified sharp
dyadic truncations; the ordinary nonlinear integral may be undefined.

\subsection{Infinite sectors}
Fix $0<\theta<2\pi$ and assume that $(K,\Phi)$ satisfies
$(\mathrm{TC}_\theta)$. We prove the sufficiency direction of
\cref{thm:sector-density}; necessity was proved in \cref{lem:tangent-concentration} for each distinct
tangent model and each sign.

Let $f\in L_c^s(\R^2;\R)$ satisfy \eqref{eq:zero-mean}.
By \cref{lem:kernel-endpoints}, $K*f\in L^p(\R^2)$ and
$\Phi(K*f)\in L^1(W_\theta)$, with both strong truncation limits.
It remains to prove the all-scale annular estimate.

Retain the two boundary rays $\Gamma_0,\Gamma_\theta$. Put
\begin{equation}\label{eq:sector-angle-constant}
 s_\theta:=\sin\!\left(\min\left\{\theta,2\pi-\theta,\frac\pi2\right\}\right)>0,
 \qquad c_\theta:=\frac14s_\theta.
\end{equation}
For $a\in\partial W_\theta\setminus\{0\}$, let
$j\in\{0,\theta\}$ be the unique index with $a\in\Gamma_j$ and set
$H_a^0:=H_j^0$. Then
\[
 B(a,c_\theta|a|)\cap\partial W_\theta
 =B(a,c_\theta|a|)\cap\Gamma_j,
 \qquad 0\notin B(a,c_\theta|a|).
\]
Consequently
\begin{equation}\label{eq:sector-local-half-plane}
 W_\theta\cap B(a,c_\theta|a|)
 =(a+H_a^0)\cap B(a,c_\theta|a|).
\end{equation}
This elementary statement also covers $\theta=\pi$ and reentrant sectors.

The local identity \eqref{eq:sector-local-half-plane} and signed
cancellation verify \eqref{eq:two-level-atomic-input} with
$D=W_\theta$, $V=\{0\}$, $c=c_\theta$, $J=\mathbb Z$, and
$\epsilon_n=0$: the ray models are exact at regular anchors, and the
origin-centered sector model is exact at every scale. Therefore
\cref{prop:two-level-annular} gives
\begin{equation}\label{eq:fine-sector}
 \sum_{n\in\mathbb Z}
 \left|\int_{W_\theta}\Phi(K_n*f)\,\dd x\right|
 \leq C_{\theta,K,\Phi}\|f\|_1^p.
\end{equation}

Here the constant may depend on $c_\theta^{-1}$; no uniformity is claimed
as $\theta\to0$ or $\theta\to2\pi$.

Together with \cref{lem:kernel-endpoints,lem:annular-summation,thm:stolyarov-mixed},
\eqref{eq:fine-sector} proves sufficiency in \cref{thm:sector-density}.

\begin{remark}
For a new datum $f\in L_c^\infty(\R^2;\R)$, now without imposing
\eqref{eq:zero-mean}, put $M_f:=\int_{\R^2}f$. If $M_f\neq0$, then, as
$|x|\to\infty$,
$K*f(x)=M_fK(x)+O_{K,f}(|x|^{\alpha-2-\gamma})$. Thus the leading nonlinear term has the
critical $|x|^{-2}$ tail, and $\Phi(K*f)$ need not be absolutely integrable
over $W_\theta$. If the vertex
angular condition
\[
 \int_{W_\theta\cap\Sph^1}
   \Phi\!\left(\sgn(M_f)\wtK(\omega)\right)\,\dd\sigma(\omega)=0
\]
holds, the leading term integrates to zero over every vertex-centered annulus, while
\eqref{eq:phi-lip} gives, for all sufficiently large $|x|$,
\[
 |\Phi(K*f(x))-\Phi(M_fK(x))|
 \leq C_{K,\Phi,f}|x|^{-2-\gamma}.
\]
Hence the radial truncations
$\int_{W_\theta\cap B_T}\Phi(K*f)$ converge as $T\to\infty$.
This is a specified improper radial limit; in general it is not an
absolutely convergent Lebesgue integral over $W_\theta$.
\end{remark}

\subsection{A fixed non-Dini domain with opposite quadratic outcomes}

The dependence on the integrand is already visible on the explicit
domain used in the convex counterexample.  Put $L(r)=\log(e/r)$ and
\begin{equation}\label{eq:pair-domain}
 \Omega_*=
 \left\{x\in B_1:
 |x|^2+\frac{2|x|}{L(|x|)}-2x_2<0\right\},
\end{equation}
where the radial expression at the origin is defined by continuity.
For the Newton kernel $K(x)=x/(2\pi|x|^2)$, set
\[
 \Phi_+(\xi)=2\xi_1\xi_2,
 \qquad \Phi_-(\xi)=\xi_1^2-\xi_2^2.
\]
Both are nonzero quadratic forms with vanishing integral on every
half-circle.

\begin{theorem}\label{thm:pair-nondini}
The domain $\Omega_*$ is bounded, strictly convex and $C^1$, is smooth
away from the origin, and does not have a Dini tangent modulus.
Nevertheless,
\begin{equation}\label{eq:pair-good-density}
 \left|\int_{\Omega_*}\Phi_+(K*f)\,\dd x\right|
 \le C\|f\|_1^2
 \qquad
 \bigl(f\in L^1(\R^2)\cap L^s_{\mathrm{loc}}(\R^2),\ s>1\bigr).
\end{equation}
For the same domain the corresponding estimate for $\Phi_-$ fails.
More precisely, for the radial unit-mass tests from
\cref{cor:convex-c1-pde-failure},
\begin{equation}\label{eq:pair-bad-density}
 \int_{\Omega_*}\Phi_-(K*\rho_\eps)\,\dd x
 =-\frac{2}{(2\pi)^2}\log\log\frac1\eps+O(1).
\end{equation}
Thus a Dini tangent modulus is not necessary for a fixed nonzero
domain--integrand pair, even when the integrand is a quadratic form.
\end{theorem}

\begin{proof}
The convexity and regularity assertions, except the failure of Dini
regularity, follow from the proof of
\cref{cor:convex-c1-pde-failure} with $\tau=e_1$.
Indeed, its defining function is strictly convex because of the
$|x|^2$ term, and is at least $1$ on $\partial B_1$; hence no part of
the boundary is created by the artificial restriction to $B_1$.
It is smooth away from $0$, and its gradient at $0$ is $-2e_2$.

We record quantitative bounds for the actual boundary in
\eqref{eq:pair-domain}.  On its right-hand branch near $0$, put
\[
 q(t)=\frac t2+\frac1{L(t)},\qquad
 X(t)=t\sqrt{1-q(t)^2},\qquad Y(t)=tq(t).
\]
Then $(X(t),Y(t))$ parametrizes that branch.  For sufficiently small
$t$, one has $X(t)\simeq t$, $X'(t)\simeq1$, and
\[
 \begin{gathered}
 q'(t)=\frac12+\frac1{tL(t)^2},\qquad
 Y'(t)=t+L(t)^{-1}+L(t)^{-2},\\
 |X''(t)|\le C\left(1+\frac1{tL(t)^3}\right),\qquad
 |Y''(t)|\le C\left(1+\frac1{tL(t)^2}\right).
 \end{gathered}
\]
Writing the boundary as $x_2=h(x_1)$, reflection symmetry makes $h$
even.  The chain rule and $L(X(t))=L(t)+O(1)$ give, for $0<|u|<u_0$,
\begin{equation}\label{eq:pair-graph-bounds}
 |h'(u)|\le\frac C{L(|u|)},\qquad
 |h''(u)|\le C\left(1+\frac1{|u|L(|u|)^2}\right),\qquad
 h'(u)\sim\frac{\operatorname{sgn}u}{L(|u|)}.
\end{equation}
In particular $h'(0)=0$, and its tangent change from $0$ is comparable
to $L(|u|)^{-1}$.  Since arclength is comparable to $|u|$, any tangent
modulus must dominate $c/L(r)$ at small $r$.  Its Dini integral
therefore diverges.

For the positive assertion define the actual circular moment
\[
 \mathcal A_+(a,r)
 =\int_{\Sph^1}\one_{\Omega_*}(a+r\omega)
                \Phi_+(K(\omega))\,\dd\sigma(\omega).
\]
We claim that there are $r_0,C>0$ such that
\begin{equation}\label{eq:pair-angular-bound}
 \sup_{a\in\partial\Omega_*}|\mathcal A_+(a,r)|
 \le C\bigl(L(r)^{-2}+r\bigr),\qquad 0<r<r_0.
\end{equation}
The integral over the tangent half-circle at any $a$ is zero, so only
the angular difference from that half-circle needs to be estimated.

First take $a=(u,h(u))$ in a fixed small graph neighborhood of $0$.
If $|u|\le4r$, all graph slopes sampled by $B(a,r)$ are at most
$C/L(r)$ by \eqref{eq:pair-graph-bounds}.  When the domain and its
tangent half-plane disagree at $a+re_t$, both the graph chord and the
tangent line force
$|\sin t|\le C L(r)^{-1}|\cos t|$.
Thus the difference is confined to two horizontal angular bands of
total length at most $C/L(r)$.  On these bands
\[
 |\Phi_+(K(e_t))|=(2\pi)^{-2}|\sin2t|\le C/L(r),
\]
and hence $|\mathcal A_+(a,r)|\le C/L(r)^2$.

If $|u|>4r$, the sampled graph stays away from $0$.
Set $M_u=C(1+(|u|L(|u|)^2)^{-1})$, which bounds $|h''|$ on the sampled
graph interval by \eqref{eq:pair-graph-bounds}. At any disagreement angle,
Taylor's formula gives
\[
 |\sin t-h'(u)\cos t|\le CrM_u\cos^2t.
\]
Since the tangent slope is bounded, this confines the difference to two
bands about the tangent directions, of total angular width at most
\[
 C r\left(1+\frac1{|u|L(|u|)^2}\right).
\]
The band width is at most $C/L(|u|)$ since $r<|u|/4$ and
the graph neighborhood is small. Together with the tangent slope
bound, this gives $|\sin2t|\le C/L(|u|)$ on those bands.  It follows that
\[
 |\mathcal A_+(a,r)|
 \le C\left(r+\frac r{|u|}L(|u|)^{-3}\right)
 \le C\bigl(r+L(r)^{-2}\bigr).
\]
For the last inequality, if $L(|u|)\ge L(r)/2$ the estimate is direct;
otherwise $r/|u|=\exp[-(L(r)-L(|u|))]\le e^{-L(r)/2}$.
Finally, the portion of $\partial\Omega_*$ outside a smaller graph
neighborhood is compact and smooth.  A finite $C^2$ atlas there gives
an angular difference of size $Cr$, uniformly in its centers.
Choose overlapping graph neighborhoods and then one common small
$r_0$ to combine these estimates.  This proves
\eqref{eq:pair-angular-bound}.

Critical homogeneity gives
\[
 \int_{\Omega_*}\Phi_+(K_n(x-a))\,\dd x
 =\int_{2^{-n-1}}^{2^{-n}}\mathcal A_+(a,r)\frac{\dd r}{r}.
\]
The right side of \eqref{eq:pair-angular-bound} is integrable against
$\dd r/r$.  The remaining bounded range of scales contributes a
finite constant, and the large annuli miss $\Omega_*$ for all boundary
centers.  Since $\Phi_+$ is even, this proves
\[
 \sum_{n\in\mathbb Z}\sup_{a\in\partial\Omega_*}
             \mathfrak a_{n,\Omega_*}(a)<\infty.
\]
Plane cancellation and \cref{thm:atomic-criterion} now imply
\eqref{eq:pair-good-density}.

For $\Phi_-$, one has $\Phi_-(e_1)=1$.  The exact domain
\eqref{eq:pair-domain} is the counterexample domain in
\cref{cor:convex-c1-pde-failure} for this tangent, so
\eqref{eq:nondini-loglog} is precisely
\eqref{eq:pair-bad-density}.  This completes the proof.
\end{proof}

\begin{remark}
The success concerns all allowed density data, not only the point
source at the origin. The small angular width is multiplied by the
zero of $\sin2t$ on the rough tangent directions.
The PDE counterpart is proved separately in \cref{cor:pair-pde}, using
the non-Dini convex-domain representation of
\cref{thm:convex-extension}.
\end{remark}
\section{Measure potentials and PDE inequalities}
\label{sec:pde-realization}
\subsection{Strong-field closure for finite measures}
\begin{lemma}\label{lem:local-measure-closure}
Let $K,\Phi$ satisfy the standing hypotheses, let $D\subset\R^2$ be
measurable, and suppose
\[
 \left|\int_D\Phi(K*h)\right|\leq A\|h\|_1^p
 \qquad(h\in C_c^\infty(\R^2;\R))
\]
for some $A\geq0$. Let $\mu$ be a compactly supported finite signed
measure and set $G=K*\mu$. Assume either that $G\in L^p(\R^2)$,
or that $D$ is bounded and $G\in L^p(O)$ on an open neighborhood
$O$ of $\overline D$. Then $\Phi(G)\in L^1(D)$ and
\begin{equation}\label{eq:local-measure-closure}
 \left|\int_D\Phi(G)\right|\leq A\|\mu\|_{\cM}^p.
\end{equation}
The same conclusion holds if the smooth estimate is assumed only for
mean-zero $h$, provided that $\mu(\R^2)=0$.
\end{lemma}

\begin{proof}
Take a nonnegative unit mollifier $\rho\in C_c^\infty(B_1)$ and put
$f_\eps=\rho_\eps*\mu$. Then $f_\eps\in C_c^\infty$,
$\|f_\eps\|_1\leq\|\mu\|_{\cM}$, and
$\int f_\eps=\mu(\R^2)$. Local Fubini gives
$K*f_\eps=\rho_\eps*G$. If $G\in L^p(\R^2)$, approximate identities
give strong $L^p(D)$ convergence to $G$. In the local case choose
$\eta\in C_c^\infty(O)$ equal to one near $\overline D$; on $D$ and
for small $\eps$, replace $G$ in this convolution by $\eta G$ and
apply the same argument. By \eqref{eq:phi-lp-continuity}, the
nonlinearities converge in $L^1(D)$, so the appropriate smooth estimate
passes to the limit.
\end{proof}

The local branch requires an $L^p$ neighborhood for convolution across
the integration boundary. In the Newton application this comes from
qualitative local finite energy, with no energy bound in terms of
$\|\mu\|_{\cM}$.

\subsection{Two-scale necessity on piecewise smooth domains}
\label{sec:c1-pde-obstruction}

Here $\Omega$ is any finitely cornered piecewise-$C^1$ domain in the
comparison class defined in the necessity subsection, including every
bounded polygon. The derivative $\partial_n$ uses the outward unit
normal. The following compactly supported double-scale test proves
necessity without imposing the Dini integral.

\begin{proposition}\label{prop:pde-tangent-test}
Let $\Phi:\R^2\to\R$ be positively two-homogeneous and Lipschitz on
$\Sph^1$. Fix an interior point, a noncorner boundary point, or a vertex
$a\in\overline\Omega$, and let $G_a^0$ be its origin-based tangent model.
There are $\eps_0>0$, numbers $L_\eps\to\infty$ as $\eps\downarrow0$, and
functions $u_{\eps,a}\in C_c^\infty(\R^2;\R)$ such that, for both signs,
\begin{equation}\label{eq:pde-tangent-log}
 \int_\Omega\Phi\bigl(\nabla(\varsigma u_{\eps,a})\bigr)\,\dd x
 =L_\eps\int_{G_a^0\cap\Sph^1}\Phi(\varsigma e)\,\dd\sigma(e)+O(1),
\end{equation}
and
\begin{equation}\label{eq:pde-tangent-rhs}
 \sup_{\substack{0<\eps<\eps_0\\\varsigma\in\{-1,1\}}}
 \left(\|\Delta(\varsigma u_{\eps,a})\|_{L^1(\Omega)}
 +\|\partial_n(\varsigma u_{\eps,a})\|_{L^1(\partial\Omega)}\right)<\infty.
\end{equation}
Consequently, the PDE inequality with $N_\Omega$ forces every
signed tangent-model cancellation throughout this piecewise-$C^1$
comparison class. Its sufficiency is not asserted on that larger class.
\end{proposition}

\begin{proof}
The local $C^1$ charts give $|\gamma(s)-a|\simeq s$,
$|(\gamma(s)-a)\cdot n(\gamma(s))|\leq Cs\omega(Cs)$, and the
area discrepancy $Cr^2\omega(Cr)$ in \cref{lem:tangent-concentration}.
At an interior point both boundary errors vanish. Choose $C_*$
absorbing these constants and $\eps_0$ so that $B(a,\eps_0)$ meets
only the local branches (inside $\Omega$ at an interior point).
With a positive comparison modulus, put
\[
 L_\eps=\omega(C_*\eps)^{-1/2},\qquad
 \delta_\eps=\eps e^{-L_\eps}<\eps/2.
\]
Here $\eps_0$ is decreased if necessary; then $L_\eps\to\infty$
and $\omega(C_*\eps)L_\eps\to0$.

Choose a nonnegative radial $\rho\in C_c^\infty(\{1/2<|x|<1\})$
with integral one, write $Q(t)=\int_{B_t}\rho$, and define
\begin{equation}\label{eq:pde-tangent-test}
 u_{\eps,a}=2\pi\Gamma*(\rho_{\delta_\eps,a}-\rho_{\eps,a}),
 \qquad w_\eps(r)=Q(r/\delta_\eps)-Q(r/\eps).
\end{equation}
The radial Newton identity used in
\cref{cor:convex-c1-pde-failure}, with its normalization at infinity, gives
$\Gamma*\rho_{t,a}=\Gamma(\,\cdot-a)$ outside $B(a,t)$.
Thus $u_{\eps,a}\in C_c^\infty(\R^2)$ vanishes outside $B(a,\eps)$,
and
\begin{equation}\label{eq:mollified-log-field}
\nabla u_{\eps,a}=w_\eps(|x-a|)\frac{x-a}{|x-a|^2}\quad(x\ne a),\qquad
 \Delta u_{\eps,a}=2\pi(\rho_{\delta_\eps,a}-\rho_{\eps,a}).
\end{equation}
Consequently $\|\Delta u_{\eps,a}\|_1\leq4\pi$.
Moreover $0\leq w_\eps\leq1$, it equals one on
$[\delta_\eps,\eps/2]$, and vanishes outside
$(\delta_\eps/2,\eps)$, so $\int w_\eps^2\,\dd r/r=L_\eps+O(1)$.

The chart bounds, truncated layer cake as in
\cref{lem:tangent-concentration}, and polar integration yield, for both signs,
\[
 \begin{aligned}
 &\|\partial_nu_{\eps,a}\|_{L^1(\partial\Omega)}\\
 &\quad+
 \left|\int_\Omega\Phi(\nabla(\varsigma u_{\eps,a}))
       -L_\eps\int_{G_a^0\cap\Sph^1}\Phi(\varsigma e)\,\dd\sigma(e)\right|
 \leq C\bigl[1+\omega(C_*\eps)(L_\eps+1)\bigr]=O(1).
 \end{aligned}
\]
Indeed, the normal error integrates $C\omega(C_*\eps)/s$.
Layer cake bounds the $|x-a|^{-2}$ integral over the symmetric
difference on $\delta_\eps/2<|x-a|<\eps$ by
$C\omega(C_*\eps)(L_\eps+1)$; the radial transitions have bounded
logarithmic length.
This proves \eqref{eq:pde-tangent-log} and \eqref{eq:pde-tangent-rhs}.
The assumed PDE bound and $L_\eps\to\infty$ force each moment to vanish.
\end{proof}

For the circularly truncated sector $\Omega_{\theta,1}$ with $a=0$, the
coefficient in \eqref{eq:pde-tangent-log} for the quadratic form
$\Phi_\theta$ from \eqref{eq:intro-vertex-form} is
\[
 \int_0^\theta\Phi_\theta(e_t)\,\dd t=\sin\theta.
\]
Thus every nonflat corner gives a nonzero coefficient; no such conclusion
is claimed for $\theta=\pi$.

\subsection{Dini layer potentials and measure extension}
\label{sec:dini-extension}

Let $\Omega$ be a finitely cornered piecewise-$C^{1,\mathrm{Dini}}$
domain, $\Sigma=\partial\Omega$, and $n$ its outward normal; holes are
allowed. Set $\Gamma(x)=(2\pi)^{-1}\log|x|$ and
\[
 \begin{gathered}
 S\lambda=\Gamma*(\lambda\,\dd s),\qquad
 k^*(x,y)=\frac{n_x\cdot(x-y)}{2\pi|x-y|^2},\\
 K^*\lambda(x)=\int_\Sigma k^*(x,y)\lambda(y)\,\dd s_y,
 \qquad T=2K^*,\quad A=-\tfrac12I+K^*.
 \end{gathered}
\]
The integrals are absolutely convergent almost everywhere by the
estimates below. Single-layer solutions of the endpoint Neumann problem
are classical; see \cite[Theorem~1 and Note~5]{MedkovaNeumann1998}
and \cite{KralMedkova1998}. With their negative fundamental solution,
the weak normal derivative has sign $-A$ in our convention
\cite[pp.~420--422 and equation~(27)]{KralMedkova1998}.
We give a direct inversion on the present class, using one $L^1/L^2$
decomposition to obtain both the quantitative $L^1$ bound and qualitative
finite energy of the same density. See also \cite{Verchota1984} for
layer-potential background.

\subsubsection{Boundary operators and endpoint inversion}

\begin{lemma}\label{lem:dini-wedge-norms}
On the two rays of a sector of opening $0<\theta<2\pi$, $T$ has zero
diagonal blocks and equal cross blocks
\begin{equation}\label{eq:dini-wedge-kernel}
 M_\theta f(s)=\frac{\sin\theta}{\pi}
 \int_0^\infty\frac{t f(t)}{s^2+t^2-2st\cos\theta}\,\dd t.
\end{equation}
On the direct sum of the ray spaces,
\begin{equation}\label{eq:dini-wedge-norms}
 \|T_\theta\|_{L^1\to L^1}=\frac{|\pi-\theta|}{\pi}<1,\qquad
 \|T_\theta\|_{L^2\to L^2}=\left|\cos\frac\theta2\right|<1.
\end{equation}
Multiplication on either side by cutoffs in $[0,1]$ does not increase
these norms.
\end{lemma}

\begin{proof}
The ray normals $(0,-1)$ and $(-\sin\theta,\cos\theta)$ give
$n_i\cdot(s\tau_i-t\tau_j)=t\sin\theta$ for $i\ne j$.
The absolute column integral of this fixed-sign kernel is
\[
 \frac{|\sin\theta|}{\pi}
 \int_0^\infty\frac{\dd r}{1+r^2-2r\cos\theta}
 =\frac{|\pi-\theta|}{\pi},
\]
which gives its exact $L^1$ norm. The isometry
$Uf(x)=e^{x/2}f(e^x)$ conjugates $M_\theta$ to convolution by
\[
 m_\theta(z)=\frac{\sin\theta}{2\pi}
             \frac{e^{-z/2}}{\cosh z-\cos\theta}.
\]
For $\theta\ne\pi$, its fixed sign makes the Fourier multiplier norm
attain $\|m_\theta\|_1$ at zero, and
\[
 \|m_\theta\|_1
 =\frac{|\sin\theta|}{\pi}
   \int_0^\infty\frac{r^{1/2}}{1+r^2-2r\cos\theta}\,\dd r
 =\left|\cos\frac\theta2\right|.
\]
Here $r=u^2$ evaluates the integral as $\pi/(2\sin(\theta/2))$.
At $\theta=\pi$ the operator is zero. The block matrix and contractive
cutoff multiplications finish the proof.
\end{proof}

\begin{lemma}\label{lem:dini-layer-decomposition}
There is one decomposition $T=B+C$ on $L^1(\Sigma)$ and $L^2(\Sigma)$
with
\[
 \|B\|_{L^1\to L^1}<1,\qquad \|B\|_{L^2\to L^2}<1,
\]
where $C$ is compact on both spaces and bounded from $L^1$ to $L^2$.
Consequently $A$ is Fredholm of index zero on $L^1(\Sigma)$.
\end{lemma}

\begin{proof}
Subdivide the boundary into finitely many short compact embedded arcs
with distinct endpoints, including closed pieces and pieces whose
original endpoints coincide. Added cuts have angle $\pi$. The arc
spaces identify with $L^p(\Sigma)$ up to null endpoints, without changing
the vertex set or tangent models.

For an arclength arc $\gamma$, the tangent modulus gives
\begin{equation}\label{eq:dini-same-sheet-kernel}
 |k^*(\gamma(s),\gamma(t))|
 \le C\frac{\omega(C|s-t|)}{|s-t|},
\end{equation}
since its numerator equals
$n(s)\cdot\int_t^s(\gamma'(a)-\gamma'(s))\,\dd a$.
At a junction $v$, write $\gamma_i(s)=v+s\tau_i+e_i(s)$, so
$|e_i(s)|\le s\omega(s)$ and $|n_i(s)-n_i(0)|\le\omega(s)$.
The separation $|s\tau_i-t\tau_j|\ge c_{\theta_v}(s+t)$ then yields
\begin{equation}\label{eq:dini-cross-sheet-kernel}
 |2k^*(\gamma_i(s),\gamma_j(t))-M_{\theta_v}(s,t)|
 \le C_{\theta_v}\frac{\omega(C(s+t))}{s+t}.
\end{equation}
For kernels localized within scale $\delta$, both absolute row and
column integrals of these errors are bounded by
\begin{equation}\label{eq:dini-layer-error}
 C_\Omega\int_0^{C_\Omega\delta}\frac{\omega(r)}r\,\dd r=o(1).
\end{equation}
Thus they are small on $L^1$ and, by Schur's test, on $L^2$.

Choose disjoint junction neighborhoods and endpoint cutoffs
$0\le\chi_v\le1$ equal to one near the junctions. The cross blocks
$\chi_v M_{\theta_v}\chi_v$ form $B_0$, with norms at most
\[
 q_1=\max_v\frac{|\pi-\theta_v|}{\pi}<1,\qquad
 q_2=\max_v\left|\cos\frac{\theta_v}2\right|<1;
\]
flat junctions contribute zero. Add to $B_0$ the cross errors with
these same cutoffs and the single-arc kernels $2k^*$ cut off to
$|s-t|<\delta$, retaining them fully for $|s-t|\le\delta/2$.
For small $\delta$, their sum $D$ has
$\|D\|_{L^p\to L^p}<(1-q_p)/2$ for $p=1,2$, so $B=B_0+D$
is contractive on both spaces.

The remainder $C=T-B$ vanishes near every single-arc diagonal and
common endpoint. Its kernel is continuous on each compact arc product,
using one-sided endpoint normals; nonincident arcs are separated.
Uniform approximation by separated continuous kernels gives finite-rank
approximation on both spaces and in $L^1\to L^2$ norm. Finally
\[
 I-T=(I-B)\bigl(I-(I-B)^{-1}C\bigr)
\]
is invertible times identity minus compact.
\end{proof}

\begin{lemma}\label{lem:dini-potential-jump}
For $\lambda\in L^1(\Sigma)$ and ambient smooth $\varphi$,
\begin{equation}\label{eq:dini-weak-green}
 \int_\Omega\nabla S\lambda\cdot\nabla\varphi\,\dd x
 =\int_\Sigma(A\lambda)\varphi\,\dd s.
\end{equation}
For $\lambda\in L^2(\Sigma)$,
$S\lambda\in C(\R^2)\cap W^{1,2}_{\rm loc}(\R^2)$ and the identity
extends to $\varphi\in W^{1,2}(\Omega)$, using its Sobolev trace.
If also $\int_\Sigma\lambda\,\dd s=0$, then
$S\lambda(x)=O(|x|^{-1})$, $\nabla S\lambda(x)=O(|x|^{-2})$ at
infinity, and $\nabla S\lambda\in L^2(\R^2)$.
\end{lemma}

\begin{proof}
At a regular boundary source $y$, Green's formula on
$\Omega\setminus\overline{B_\eps(y)}$ and the limiting inner
semicircle contribution $-\varphi(y)/2$ give
\[
 \int_\Omega\nabla\Gamma(x-y)\cdot\nabla\varphi(x)\,\dd x
 =\int_\Sigma k^*(x,y)\varphi(x)\,\dd s_x-\tfrac12\varphi(y).
\]
The kernel bounds above imply
$\operatorname*{ess\,sup}_{y\in\Sigma}\int_\Sigma|k^*(x,y)|\,\dd s_x<\infty$;
also $\sup_y\int_\Omega|x-y|^{-1}\,\dd x<\infty$.
Fubini therefore permits integration against $\lambda$ and proves
\eqref{eq:dini-weak-green}, including the distributional gradient.

The bound $s(\Sigma\cap B_r(z))\le Cr$ gives
\[
 \sup_z\int_{\Sigma\cap B_r(z)}|\log|z-y||^2\,\dd s_y
 \le Cr(1+|\log r|^2).
\]
For $\lambda\in L^2$, Cauchy--Schwarz and continuity of the kernel
away from zero prove continuity of $S\lambda$. On any ball containing
$\Sigma$, the trace inequality places $\lambda\,\dd s$ in $W^{-1,2}$.
A weak Dirichlet solution $z$ of $\Delta z=\lambda\,\dd s$ belongs
to $W^{1,2}_0$; since $S\lambda-z$ is harmonic, this proves local
finite energy. The trace theorem and density extend
\eqref{eq:dini-weak-green} to $W^{1,2}$ tests.
Subtracting the zero-mass term $\Gamma(x)\int\lambda$ and expanding
at infinity gives the last assertions.
\end{proof}

\begin{lemma}\label{lem:dini-prescribed-mass}
Put $L^1_0(\Sigma)=\{q\in L^1(\Sigma):\int_\Sigma q\,\dd s=0\}$.
For each $(q,m)\in L^1_0(\Sigma)\times\R$ there is a unique
$\lambda\in L^1(\Sigma)$ satisfying
\[
 A\lambda=q,\qquad \int_\Sigma\lambda\,\dd s=m,
 \qquad\|\lambda\|_1\le C_\Omega(\|q\|_1+|m|).
\]
The solution depends linearly on $(q,m)$ and belongs to $L^2$ if
$q\in L^2$. No quantitative $L^2$ bound in terms of these $L^1$ data
is asserted.
\end{lemma}

\begin{proof}
If $A\lambda=0$, then $\lambda=(I-B)^{-1}C\lambda\in L^2$:
the same Neumann series converges in both spaces, and finite boundary
length identifies its $L^2$ and $L^1$ limits. Testing
\eqref{eq:dini-weak-green} with $1$ gives
$\operatorname{ran}A\subset L^1_0$.

If also $\int\lambda=0$, testing with $S\lambda$ makes it a constant
$c$ on connected $\Omega$ and, by continuity, on all of $\Sigma$.
Now test the global equation $\Delta S\lambda=\lambda\,\dd s$ with
$\zeta_R S\lambda$, using local finite energy and trace duality, where $0\le\zeta_R\le1$ equals one on $B_R$, vanishes outside
$B_{2R}$, and $|\nabla\zeta_R|\lesssim R^{-1}$. The preceding
far-field bounds make the cutoff error tend
to zero, yielding
\begin{equation}\label{eq:layer-global-energy}
 \int_{\R^2}|\nabla S\lambda|^2\,\dd x
 =-\int_\Sigma\lambda S\lambda\,\dd s
 =-c\int_\Sigma\lambda\,\dd s=0.
\end{equation}
Thus $S\lambda=0$ and $\lambda=0$. This argument includes all holes.
The mass functional is therefore injective on $\ker A$, and index
zero gives
$1\le\operatorname{codim}\operatorname{ran}A=\dim\ker A\le1$.
Consequently $\operatorname{ran}A=L^1_0$, and a kernel generator of
nonzero mass adjusts any solution to mass $m$. The bounded inverse
theorem applied to $\lambda\mapsto(A\lambda,\int\lambda)$ gives the
estimate and linearity. For $q\in L^2$, the identity
$(I-B)\lambda=C\lambda-2q$ gives the same $L^2$ upgrade.
\end{proof}

\subsubsection{The potential representation and compact extension}

\begin{lemma}\label{lem:dini-volume-trace}
Uniformly for $y\in\Omega$,
\begin{equation}\label{eq:dini-volume-kernel}
 \int_\Sigma\frac{|n_x\cdot(x-y)|}{|x-y|^2}\,\dd s_x\le C_\Omega.
\end{equation}
For bounded $f$ on $\Omega$ and $v=\Gamma*(f\mathbf1_\Omega)$,
$\nabla v$ is continuous, $\partial_nv\in L^2(\Sigma)$, and
\begin{equation}\label{eq:dini-volume-trace}
 \|\partial_nv\|_1\le C_\Omega\|f\|_1,\qquad
 \int_\Sigma\partial_nv\,\dd s=\int_\Omega f\,\dd x.
\end{equation}
The weak Green identity holds for $W^{1,2}(\Omega)$ tests.
\end{lemma}

\begin{proof}
Extend each endpoint graph by its tangent line, preserving the Dini
control. In a patch $x=(t,h(t))$ with $|h'|\le1/10$, put
$d=|h(a)-b|$ and $r=|t-a|$ for the source $(a,b)$. Then
\[
 |h(t)-b-h'(t)(t-a)|\le d+Cr\omega(Cr),\qquad
 |(t,h(t))-(a,b)|^2\ge c(r^2+d^2).
\]
The graph normal normalization cancels arclength, so the patch integral
is at most
\[
 C\int\frac{d+r\omega(Cr)}{r^2+d^2}\,\dd t
 \le C\pi+C\int_0^{Cr_0}\frac{\omega(r)}r\,\dd r.
\]
Sources outside the extended graph interval are separated from the
smaller patch; a finite cover proves \eqref{eq:dini-volume-kernel}.
Translation continuity of the locally integrable kernel $\nabla\Gamma$
gives continuity and local boundedness of $\nabla v$. Tonelli yields
the $L^1$ trace bound, and finite boundary length gives $L^2$ membership.
Mollification converges uniformly for the gradients near
$\overline\Omega$ and in $L^2$ for the Laplacians. The divergence
theorem and Green identity pass to the limit, then extend to
$W^{1,2}$ tests by the trace theorem and density.
\end{proof}

\begin{proposition}\label{prop:dini-potential-representation}
For every $u\in C^\infty(\overline\Omega;\R)$ there is a finite signed
measure, depending linearly on $u$,
\begin{equation}\label{eq:dini-potential-measure}
 \mu=(\Delta u)\mathbf1_\Omega\,\mathcal L^2+\lambda\,\dd s,
 \qquad\lambda\in L^1(\Sigma)\cap L^2(\Sigma),
\end{equation}
such that
\begin{equation}\label{eq:dini-potential-budget}
 \mu(\R^2)=0,\qquad \|\mu\|_{\cM}\le C_\Omega N_\Omega(u).
\end{equation}
Its potential $V=\Gamma*\mu$ is continuous, belongs to
$W^{1,2}_{\rm loc}(\R^2)$, has $\nabla V\in L^2(\R^2)$, and satisfies
\begin{equation}\label{eq:dini-potential-matching}
 u=u_\Omega+(V-V_\Omega)\quad\text{in }\Omega.
\end{equation}
The finite-energy assertion is qualitative.
\end{proposition}

\begin{proof}
Put $f=\Delta u$, $g=\partial_nu$, and $v=\Gamma*(f\mathbf1_\Omega)$.
The divergence theorem and \cref{lem:dini-volume-trace} give
\[
 q=g-\partial_nv\in L^1_0(\Sigma)\cap L^2(\Sigma),\qquad
 m=-\int_\Omega f\,\dd x,\qquad
 \|q\|_1+|m|\le C_\Omega N_\Omega(u).
\]
Take $\lambda$ from \cref{lem:dini-prescribed-mass}. The measure
\eqref{eq:dini-potential-measure} has the required mass and variation.
By \cref{lem:dini-potential-jump}, $V=v+S\lambda$ is continuous and
locally of finite energy. Zero total mass gives
$\nabla V(x)=O(|x|^{-2})$ at infinity, hence global finite energy.
In $\Omega$, its Laplacian is $f$ and its weak normal derivative is
$\partial_nv+A\lambda=g$. Testing the weak equation for $u-V$ with
$u-V$ makes this difference constant; connectedness and averaging give
\eqref{eq:dini-potential-matching}. All choices are linear in $u$.
\end{proof}

\begin{proof}[Proof of \cref{thm:global-extension}]
Take $\mu,V$ from \cref{prop:dini-potential-representation}, fix a ball
$B\supset\overline\Omega$ and $\chi\in C_c^\infty(B)$ equal to one
near $\overline\Omega$, and set
$E_\Omega u=u_\Omega+\chi(V-V_\Omega)$. It has the stated agreement,
continuity, finite energy, support, and constant-preserving properties.
Since $\supp\mu\subset\overline\Omega$,
\[
 \Delta E_\Omega u
 =\mu+\bigl(2\nabla\chi\cdot\nabla V+(V-V_\Omega)\Delta\chi\bigr)
       \mathcal L^2.
\]
Local integrability of $|x|^{-1}$, Tonelli, and Poincar\'e on $B$ give
\[
 \|\nabla V\|_{L^1(B)}+\|V-V_\Omega\|_{L^1(B)}
 \le C_{B,\Omega}\|\mu\|_{\cM}.
\]
Here the change from $V_B$ to $V_\Omega$ uses
$|V_B-V_\Omega|\le|\Omega|^{-1}\|V-V_B\|_{L^1(\Omega)}$.
Together with \eqref{eq:dini-potential-budget}, the product formula proves
\eqref{eq:global-extension}, with all cutoffs fixed independently of $u$.
\end{proof}

\begin{corollary}\label{thm:compact-sector-extension}
For every $0<\theta<2\pi$ there is $C_\theta<\infty$ such that, for
each fixed $\varrho>0$, a linear operator
$E^{\mathrm{loc}}_{\theta,\varrho}$ on
\[
 \mathcal X_{\theta,\varrho}
 :=\{V|_{W_\theta}:V\in C_c^\infty(\R^2;\R),\ \supp V\subset B_\varrho\}
\]
satisfies
\[
 E^{\mathrm{loc}}_{\theta,\varrho}u=u\ \text{in }W_\theta,\qquad
 E^{\mathrm{loc}}_{\theta,\varrho}u\in C_c(\R^2)\cap W^{1,2}(\R^2),
 \qquad\supp E^{\mathrm{loc}}_{\theta,\varrho}u\subset B_{3\varrho},
\]
and
\begin{equation}\label{eq:compact-sector-extension}
 \|\Delta(E^{\mathrm{loc}}_{\theta,\varrho}u)\|_{\cM(\R^2)}
 \le C_\theta\bigl(\|\Delta u\|_{L^1(W_\theta)}
                 +\|\partial_nu\|_{L^1(\partial W_\theta)}\bigr).
\end{equation}
After translation and rotation the same statement holds, with support
in $B(a,3\varrho)$ when the vertex is $a$.
\end{corollary}

\begin{proof}
For $\varrho=1$, apply \cref{prop:dini-potential-representation} to
$u|_D$, where $D=W_\theta\cap B_4$, and put $H=u_D+V-V_D$.
Thus $H=u$ on $D$, $\Delta H=\mu$, and
$\|\nabla H\|_{L^1(B_3)}\le C\|\mu\|_{\cM}$ by the local kernel bound.
Since $H=0$ on the fixed positive-area set $D\cap(B_3\setminus B_1)$,
anchored Poincar\'e also gives
$\|H\|_{L^1(B_3)}\le C_\theta\|\mu\|_{\cM}$.
Choose a fixed radial $\eta\in C_c^\infty(B_3)$ equal to one on $B_1$
and set $E^{\mathrm{loc}}_{\theta,1}u=\eta H$.
It agrees with $u$ throughout $W_\theta$, since both vanish outside
$B_1$ there. Its product Laplacian has variation at most
$C_\theta\|\mu\|_{\cM}\le C_\theta N_D(u)$.
The artificial circular boundary contributes nothing, so $N_D(u)$ is
exactly the budget in \eqref{eq:compact-sector-extension}.
Linearity follows from the fixed choices. Dilation gives general
$\varrho$: in two dimensions both $L^1$ data norms and Laplacian total
variation are invariant. Translations and rotations preserve them.
\end{proof}

\subsection{Strictly convex \texorpdfstring{$C^1$}{C1} domains without a Dini assumption}

Convexity supplies an endpoint extension mechanism different from the
Dini perturbation argument. In the following statement, strict convexity
means that the boundary contains no nontrivial line segment. No positive
lower bound for curvature is assumed.

\begin{theorem}\label{thm:convex-extension}
Let $\Omega\subset\R^2$ be a bounded strictly convex domain with $C^1$
boundary. Then the potential representation of
\cref{prop:dini-potential-representation} and the constant-preserving
extension of \cref{thm:global-extension} hold on $\Omega$, with the same
$L^1$ boundary-value budget and qualitative finite-energy assertions.
In particular, every $u\in C^\infty(\overline\Omega;\R)$ has a compactly
supported measure $\mu$, depending linearly on $u$, such that
\[
 \mu(\R^2)=0,\qquad
 \|\mu\|_{\cM}\le C_\Omega N_\Omega(u),\qquad
 K*\mu=\nabla u\quad\hbox{in }\Omega,
 \qquad K(x)=\frac{x}{2\pi|x|^2},
\]
and $\Gamma*\mu$ is continuous and locally of finite energy, with its
gradient in $L^2(\R^2)$. No Dini condition is required.
\end{theorem}

\begin{proof}
Use the layer notation $S,K^*,T=2K^*$ and $A=-\tfrac12I+K^*$ above.
Let $\omega_\tau$ be a common modulus of continuity of the unit tangent
on the compact $C^1$ curve $\Sigma=\partial\Omega$. Convexity gives
$k^*(x,y)\ge0$. For a fixed source $y\in\Sigma$, an anticlockwise
arclength parametrization gives
\[
 k^*(x,y)\,\dd s_x=\frac1{2\pi}\,\dd\arg(x-y).
\]
The angle is monotone along the curve with $y$ removed, with total
variation $\pi$. Hence
\begin{equation}\label{eq:convex-column-mass}
 \int_\Sigma T(x,y)\,\dd s_x=1,\qquad
 \sup_{y\in\Sigma}\int_{|x-y|<r}T(x,y)\,\dd s_x
 \le C\omega_\tau(Cr)\longrightarrow0.
\end{equation}
For the second assertion, use uniform small $C^1$ graph charts at $y$.
On each side of $y$, all chord directions at distance below $r$ differ
from the corresponding one-sided tangent direction by at most
$C\omega_\tau(Cr)$. Monotonicity turns this angular variation directly
into the displayed integral, without integrating $\omega_\tau(r)/r$.

A smooth removal of the diagonal makes the kernel of $T$ continuous on
$\Sigma\times\Sigma$, and \eqref{eq:convex-column-mass} makes the removed
part tend to zero in $L^1$ operator norm. Thus $T$ is compact on $L^1$,
is positive, and preserves integrals. Moreover strict convexity makes
$T(x,y)>0$ whenever $x\ne y$. If $T\lambda=\lambda$ and
$\int_\Sigma\lambda=0$, a nonzero real $\lambda$ would have nonzero
positive and negative parts. Strict positivity then yields
$|T\lambda|<T|\lambda|$ almost everywhere. Integration contradicts
$\|T\lambda\|_1=\|\lambda\|_1$. Therefore the mass functional is
injective on $\ker A$. The compact Fredholm alternative and
$\operatorname{ran}A\subset L^1_0(\Sigma)$ now give
\begin{equation}\label{eq:convex-prescribed-mass}
 \lambda\longmapsto\left(A\lambda,\int_\Sigma\lambda\,\dd s\right)
 :L^1(\Sigma)\longrightarrow L^1_0(\Sigma)\times\R
\end{equation}
as a bounded isomorphism: its kernel is trivial, and
$1\le\operatorname{codim}\operatorname{ran}A=\dim\ker A\le1$.
Consequently its inverse obeys
$\|\lambda\|_1\le C_\Omega(\|q\|_1+|m|)$.

We use one classical input to identify the same density in $L^2$:
the $L^2$ Neumann layer-potential theorem, including boundedness of
$K^*:L^2(\Sigma)\to L^2(\Sigma)$, on bounded Lipschitz domains
with connected boundary gives, for every $q\in L^2_0(\Sigma)$, a
$\lambda_0\in L^2_0(\Sigma)$ with $A\lambda_0=q$;
see \cite{Verchota1984}. On $L^2(\Sigma)$ its principal-value
operator agrees with the present $L^1$ realization: truncations converge
in $L^1$ by \eqref{eq:convex-column-mass} and in $L^2$ by the
classical layer-potential theorem, so their limits agree. Its sign is
adjusted to our choice
$\Gamma=(2\pi)^{-1}\log|x|$. This input applies because a bounded
convex $C^1$ domain has connected Lipschitz boundary. To obtain a kernel
element of mass one in $L^2$, put $b=|\Sigma|^{-1}$ and solve
$A\eta=Ab$ with $\eta\in L^2_0(\Sigma)$; the datum has zero integral.
Then $e=b-\eta\in\ker A\cap L^2$ and $\int e=1$.
Thus $\lambda_0+me$ is an $L^2$ solution with prescribed mass $m$.
The uniqueness in \eqref{eq:convex-prescribed-mass} identifies it with
its $L^1$ solution. In particular, the endpoint bound and qualitative
energy concern one and the same density.

The weak jump identity and the continuity and local finite energy of
$S\lambda$ for $\lambda\in L^2$ follow by the proof of
\cref{lem:dini-potential-jump}: the required column bound is now
\eqref{eq:convex-column-mass}, the small-circle limit uses only $C^1$
regularity, and the remaining logarithmic-kernel and trace estimates
require only a Lipschitz curve. For a volume source $y\in\Omega$,
convexity and the divergence theorem give
\[
 \partial_{n_x}\Gamma(x-y)\ge0,\qquad
 \int_\Sigma\partial_{n_x}\Gamma(x-y)\,\dd s_x=1.
\]
Hence, for $v=\Gamma*(f\mathbf1_\Omega)$ with bounded $f$,
$\|\partial_nv\|_1\le\|f\|_1$; its continuous gradient also gives an
$L^2$ trace. With $f=\Delta u$, solve
\[
 q=\partial_nu-\partial_nv,\qquad m=-\int_\Omega f,
 \qquad \mu=f\mathbf1_\Omega\,\mathcal L^2+\lambda\,\dd s.
\]
All the hypotheses of the proof of
\cref{prop:dini-potential-representation} now hold, so that proof gives
zero mass, the stated budget, and the potential matching. The same fixed
cutoff argument used for \cref{thm:global-extension} supplies the extension.
\end{proof}

\begin{corollary}\label{cor:convex-density-to-pde}
Let $\Omega$ be as in \cref{thm:convex-extension} and let
$\Phi:\R^2\to\R$ be positively two-homogeneous and angularly Lipschitz.
If its Newton-kernel density inequality holds on $\Omega$, then
\[
 \left|\int_\Omega\Phi(\nabla u)\,\dd x\right|^{1/2}
 \le C_{\Omega,\Phi}N_\Omega(u)
 \qquad(u\in C^\infty(\overline\Omega;\R)).
\]
\end{corollary}

\begin{proof}
Apply \cref{thm:convex-extension} and the local strong-field closure
in \cref{lem:local-measure-closure} to its representing measure.
\end{proof}
\subsection{PDE criteria and the compact sector transfer}
\begin{proof}[Proof of \cref{thm:pde-mazya}]
Homogeneity identifies \textup{(i)} with the Newton-kernel
cancellation in \cref{thm:curved-density}.
Assume \textup{(i)}. By \cref{prop:dini-potential-representation},
there is a compactly supported finite measure $\mu$ with
\[
 \|\mu\|_{\cM}\leq C_\Omega N_\Omega(u),\qquad
 K*\mu=\nabla u\quad\hbox{a.e. in }\Omega.
\]
The potential $V=\Gamma*\mu$ belongs to
$W^{1,2}_{\mathrm{loc}}(\R^2)$, so $K*\mu=\nabla V\in L^2(O)$
on a bounded open neighborhood $O$ of $\overline\Omega$. This
verifies the neighborhood hypothesis in \cref{lem:local-measure-closure}.
The $L^2$ norm is needed only for the limiting passage; the
quantitative bound uses solely $\|\mu\|_{\cM}$. The density estimate
and that lemma, with $p=2$, therefore yield
\[
 \left|\int_\Omega\Phi(\nabla u)\,\dd x\right|
 \leq C_{\rm dens}\|\mu\|_{\cM}^2
 \leq C_{\rm dens}C_\Omega^2N_\Omega(u)^2.
\]
Taking square roots proves \textup{(ii)}.
Conversely, \cref{prop:pde-tangent-test} supplies tests with $N_\Omega$ bounded as $\eps\downarrow0$
at every interior, noncorner boundary, and vertex point of a
piecewise-$C^1$ domain. Their leading
logarithmic coefficient is the corresponding signed tangent moment.
Condition \textup{(ii)} forces each coefficient to vanish, proving
\textup{(i)}.
\end{proof}

\begin{corollary}
\label{cor:sector-newton-pde}
Let $0<\theta<2\pi$, and let $\Phi:\R^2\to\R$ be positively
two-homogeneous and Lipschitz on $\Sph^1$. Signed cancellation
\[
 \int_{G^0\cap\Sph^1}\Phi(\varsigma e)\,\dd\sigma(e)=0
 \qquad(G^0\in\mathfrak T_\theta,\ \varsigma\in\{-1,1\})
\]
is equivalent to the existence of $C_{\theta,\Phi}<\infty$ such that
every $u\in C_c^\infty(\R^2;\R)$ satisfies
\[
 \left|\int_{W_\theta}\Phi(\nabla u)\,\dd x\right|^{1/2}
 \leq C_{\theta,\Phi}
 \bigl(\|\Delta u\|_{L^1(W_\theta)}
       +\|\partial_nu\|_{L^1(\partial W_\theta)}\bigr).
\]
The constant is independent of the support location and radius. The
set $\mathfrak T_\theta$ contains only distinct models, including when
$\theta=\pi$.
\end{corollary}

\begin{proof}
Assume cancellation and choose $R$ with $\supp u\subset B_R$.
The compact extension in \cref{thm:compact-sector-extension} gives
$F=u$ on $W_\theta$, $F\in C_c\cap W^{1,2}(\R^2)$, and
$\mu=\Delta F$ with
\[
 \|\mu\|_{\cM}\leq C_\theta
 \bigl(\|\Delta u\|_{L^1(W_\theta)}
       +\|\partial_nu\|_{L^1(\partial W_\theta)}\bigr).
\]
Compact support gives $\mu(\R^2)=0$ and, in distributions,
$K*\mu=K*\Delta F=(\Delta K)*F=\nabla F$, since
$\Delta K=\nabla\delta_0$. Both fields are locally integrable, and
$\nabla F\in L^2(\R^2)$. Apply \cref{thm:sector-density} and the
global, mean-zero branch of \cref{lem:local-measure-closure} to get
\[
 \left|\int_{W_\theta}\Phi(\nabla u)\right|
 \leq C_{\theta,\Phi}\|\mu\|_{\cM}^2.
\]
Taking square roots proves sufficiency, uniformly in $R$.
Conversely, fix a representative point of each distinct
tangent model and a larger truncated sector $W_\theta\cap B_T$.
The tests in \cref{prop:pde-tangent-test} are supported away from its
outer circular arc. Their integrals and $L^1$ data bounds therefore agree with
those on $W_\theta$, so the assumed estimate forces every signed
angular coefficient to vanish.
\end{proof}

\subsection{The non-Dini pair for the PDE inequality}
\begin{corollary}\label{cor:pair-pde}
On the fixed domain $\Omega_*$ of \cref{eq:pair-domain}, every
$u\in C^\infty(\overline\Omega_*;\R)$ satisfies
\[
 \left|\int_{\Omega_*}2u_{x_1}u_{x_2}\,\dd x\right|^{1/2}
 \le C_{\Omega_*}N_{\Omega_*}(u).
\]
The corresponding inequality with integrand
$u_{x_1}^2-u_{x_2}^2$ fails. Thus a Dini tangent modulus is not
necessary for a fixed nonzero PDE domain--integrand pair either.
\end{corollary}
\begin{proof}
The domain is strictly convex and $C^1$ by \cref{thm:pair-nondini}.
Its positive density result and \cref{cor:convex-density-to-pde}
give the stated estimate. The negative conclusion follows from
\cref{cor:convex-c1-pde-failure} with $\Phi=\Phi_-$ and horizontal
tangent, whose tests have $N_{\Omega_*}(u_\varepsilon)\le2$.
\end{proof}
\section{Quadratic testing and polygonal consequences}
\label{sec:quadratic-results}
\subsection{Proof of the signed quadratic criterion}
\label{subsec:quad-cumulative-criterion}

We prove \cref{thm:quad-cumulative-criterion}.

\begin{proof}[Proof of \cref{thm:quad-cumulative-criterion}]
Suppose first that \eqref{eq:quad-cumulative-density} holds. Choose once
and for all a nonnegative radial $\rho\in C_c^\infty(B_1)$ of integral one, and test
with $f=\rho_{\varepsilon,a}$. The radial Newton identity gives
$K*\rho_{\varepsilon,a}(x)=K(x-a)$ for $|x-a|\geq\varepsilon$.
Moreover $\|K*\rho_{\varepsilon,a}\|_\infty\leq C\varepsilon^{-1}$,
so the absolute integral of its quadratic form over
$\Omega\cap B(a,\varepsilon)$ is at most $C\|A\|_{\rm op}$.
The testing inequality gives the first bound, uniformly in $a$ and
$\varepsilon$. A fixed universal choice of $\rho$ makes its constant
numerical; the qualitative last testing assertion holds for any fixed
radial unit-mass choice.

Conversely, for $y\ne z$ define the absolutely convergent integral
\[
 B_A(y,z):=\int_\Omega K(x-y)^{\mathsf T}A K(x-z)\,\dd x,
 \qquad d=|y-z|.
\]
On $B(y,2d)$, scaling and rotation give
\[
 \int_{B(y,2d)}|K(x-y)|\,|K(x-z)|\,\dd x
 \leq C\int_{B_2}\frac{\dd u}{|u|\,|u-e_1|}\leq C.
\]
On its complement,
$|K(x-z)-K(x-y)|\leq Cd|x-y|^{-2}$. Hence
\begin{equation}\label{eq:quad-cumulative-bipole}
 \left|B_A(y,z)-
 \int_{\Omega\setminus B(y,2d)}\Phi_A(K(x-y))\,\dd x\right|
 \leq C\|A\|_{\rm op}
       \left(1+d\int_{2d}^\infty r^{-2}\,\dd r\right)
 \leq C\|A\|_{\rm op}.
\end{equation}
Thus $|B_A(y,z)|\leq\mathcal P_A(\Omega)+C\|A\|_{\rm op}$
for all $y\ne z$.

For compactly supported $f\in L^s$, $s>1$,
\cref{lem:kernel-endpoints}, applied also to the scalar kernel $|K|$,
gives $|K|*|f|\in L^2(\Omega)$. In particular,
\[
 \int_\Omega\iint |f(y)f(z)|\,|K(x-y)|\,|K(x-z)|
                    \,\dd y\dd z\dd x<\infty.
\]
Absolute Fubini and the quadratic expansion therefore yield
\[
 \int_\Omega\Phi_A(K*f)\,\dd x
 =\iint f(y)f(z)B_A(y,z)\,\dd y\dd z.
\]
The source diagonal has Lebesgue product measure zero and adds
no term. The uniform bound for $B_A$ proves the second estimate for
these $f$.

For general data in \eqref{eq:quad-cumulative-density}, take
$f_R=f\one_{B_R}$. If $\Omega\subset B_{R_0}$ and $R>R_0+1$, then
\[
 \|K*(f-f_R)\|_{L^\infty(\Omega)}
 \leq\frac{\|f-f_R\|_1}{2\pi(R-R_0)}\longrightarrow0.
\]
Together with a fixed compact truncation and
\cref{lem:kernel-endpoints}, this gives $K*f_R\to K*f$ in
$L^2(\Omega)$. The quadratic integrals converge by
\eqref{eq:phi-lp-continuity}, proving the full statement.
\end{proof}

The conclusion is specific to quadratic integrands: the proof replaces
arbitrary source interactions by the bilinear kernel
\eqref{eq:quad-cumulative-bipole}. For a general homogeneous $\Phi$,
concentrated-source tests give a necessary cumulative bound, but the
bilinear argument above is unavailable. Accordingly, the quadratic
criterion does not assert a signed cumulative characterization for all
nonlinearities.
\subsection{Vertex rigidity and polygonal cancellation spaces}
We apply the density and PDE criteria to
\begin{equation}\label{eq:newton-kernel}
 K(x)=\nabla\!\left(\frac1{2\pi}\log|x|\right)
 =\frac{x}{2\pi|x|^2},\qquad \alpha=1,\quad p=2.
\end{equation}
For a bounded finitely cornered piecewise-$C^{1,\mathrm{Dini}}$ domain
$\Omega$, put
\[
 \Phi_A(\xi)=\xi^{\mathsf T}A\xi,\qquad
 \mathcal Q(\Omega)=\{A=A^{\mathsf T}\in\R^{2\times2}:(K,\Phi_A)
 \text{ satisfies \eqref{eq:signed-model-cancel}}\}.
\]
We call $\Omega$ \emph{quadratically rigid} if $\mathcal Q(\Omega)=\{0\}$.
Let $b_v\in\Sph^1$ span the internal bisector of the vertex cone at $v$, and let
$\mathfrak B(\Omega)=\{[b_v]:v\in\mathcal V(\Omega)\}\subset\mathbb{RP}^1$;
brackets denote unoriented lines. The PDE proofs below do not use this
classification.

\begin{corollary}
\label{cor:quadratic-rigidity}\label{cor:quadratic}\label{cor:quadratic-dimension}
The density and PDE inequalities of
\cref{thm:curved-density,thm:pde-mazya} hold for $\Phi_A$ exactly when
\[
 A\in\mathcal Q(\Omega)
 =\{A=A^{\mathsf T}:\operatorname{tr}A=0,
                    \ \Phi_A(b_v)=0\text{ for every }v\in\mathcal V(\Omega)\}.
\]
Without vertices this space has dimension two. With vertices it has
dimension one if $\mathfrak B(\Omega)$ lies in some orthogonal pair
$\{L,L^\perp\}$, and is zero otherwise. Thus three distinct bisector
lines force rigidity; every nondegenerate triangle is rigid, whereas
every parallelogram has a one-dimensional admissible space. For an
axis-parallel orthogonal polygon the latter is
$\Phi_A(\xi)=c(\xi_1^2-\xi_2^2)$, $c\in\R$.
\end{corollary}

\begin{proof}
The criteria reduce the claim to angular integrals; evenness removes
the kernel sign and the factor $(2\pi)^{-2}$ is immaterial. Since
$\int_0^{2\pi}\Phi_A(e_t)\,\dd t=\pi\operatorname{tr}A$, the plane
forces zero trace. Then $\Phi_A(e_t)=a\cos2t+b\sin2t$ and
\[
 \int_\gamma^{\gamma+\theta}\Phi_A(e_t)\,\dd t
 =\sin\theta\,\Phi_A(e_{\gamma+\theta/2}).
\]
Half-plane cancellation is automatic; at a nonflat vertex
$\sin\theta_v\ne0$, giving exactly $\Phi_A(b_v)=0$.
A nonzero second harmonic has one orthogonal pair of projective zero
lines and is determined up to scale by that pair. This proves the
dimension claims. Triangle bisectors are distinct; parallelogram
bisectors are parallel to $u+v,u-v$ for unit side directions $u,v$;
orthogonal-polygon bisectors are diagonal. These give the examples.
\end{proof}

\begin{theorem}
\label{thm:quadratic-polygon-moduli}
Let $\mathscr P_N$, $N\geq3$, be the space of labelled simple nonflat
$N$-gons in cyclic vertex order, modulo orientation-preserving
similarities. Normalize $z_1=0,z_2=1$ in $\mathbb C$, identifying
$\mathscr P_N$ with an open subset of $\R^{2N-4}$ and using its
Lebesgue measure. Cyclically set $z_{N+1}=z_1$, $\tau_0=\tau_N$, and
\begin{equation}\label{eq:polygon-moduli-zeta}
 e_i=z_{i+1}-z_i,\qquad \tau_i=e_i/|e_i|,\qquad
 \zeta_i=(\tau_{i-1}\tau_i)^2\in\Sph^1.
\end{equation}
Then
\begin{equation}\label{eq:polygon-moduli-criterion}
 \mathcal Q(P)\ne\{0\}\quad\Longleftrightarrow\quad
 \zeta_1=\cdots=\zeta_N,
\end{equation}
and $\dim\mathcal Q(P)=1$ in this case. The nonrigid locus
$\mathscr R_N=\{P\in\mathscr P_N:\mathcal Q(P)\ne\{0\}\}$ is relatively
closed, nowhere dense and Lebesgue null. More precisely:
\begin{enumerate}[label=\textup{(\roman*)},leftmargin=2.2em]
\item $\mathscr R_N=\varnothing$ for odd $N$.
\item For $N=2m$, nonrigidity is equivalent to alternating between two
distinct projective edge directions:
\[
 [\tau_{2j-1}]=U,\qquad [\tau_{2j}]=V\qquad (j=1,\ldots,m).
\]
For any unit representatives $u,v$ of $U,V$,
\begin{equation}\label{eq:polygon-admissible-matrix}
 \mathcal Q(P)=\operatorname{span}\{u\otimes u-v\otimes v\},
 \qquad u\otimes u=uu^{\mathsf T}.
\end{equation}
Locally within each edge-sign chamber
$\tau_{2j-1}=\pm u$, $\tau_{2j}=\pm v$, the locus $\mathscr R_N$ is
smoothly embedded with dimension and codimension both $N-2$.
\item $\mathscr R_4$ is exactly the parallelograms. Among all convex
simple nonflat polygons, these are the only nonrigid ones.
\end{enumerate}
\end{theorem}

\begin{proof}
The vertex rays are $-\tau_{i-1},\tau_i$, so either bisector has
$b_i^4=(\tau_{i-1}\tau_i)^2=\zeta_i$. Thus $\zeta_i$ records the
orthogonal bisector pair, including at reentrant vertices.
\Cref{cor:quadratic-rigidity} proves \eqref{eq:polygon-moduli-criterion},
and equality of successive pairs is precisely
\begin{equation}\label{eq:polygon-edge-recurrence}
 \zeta_{i+1}/\zeta_i=(\tau_{i+1}/\tau_{i-1})^2=1,
 \qquad [\tau_{i+1}]=[\tau_{i-1}].
\end{equation}
For odd $N$ this forces all edge lines to coincide, contradicting
nonflatness; for even $N$ it is the claimed alternation.
The nonzero trace-free matrix $u\otimes u-v\otimes v$ vanishes on
$u+v,u-v$, the bisector pair, proving
\eqref{eq:polygon-admissible-matrix}. The expression is unchanged by
changing either unit representative's sign.

For the dimension, $e_1=1$ and the unique upper-half-circle unit
representative $q=e^{i\vartheta}$ of $V$ gives
\[
 e_{2j-1}=a_j,\qquad e_{2j}=b_jq,\qquad a_1=1,
 \qquad \sum_j a_j=\sum_j b_j=0.
\]
Here $0<\vartheta<\pi$ and the real nonzero $a_j,b_j$ have locally
fixed signs. These parameters have dimension
$1+(m-2)+(m-1)=N-2$. Cumulative edge summation is a smooth embedding:
its inverse recovers $q$ from $[e_2]$, then the real coefficients.
Simplicity and nonflatness are open conditions. Nonemptiness for every
even $N$ follows from a rectangle by replacing a corner with a small
orthogonal staircase, adding two vertices each time.
The equality in \eqref{eq:polygon-moduli-criterion} gives relative
closedness. Positive codimension in these charts, together with second
countability, gives nullity and hence nowhere density.

For $N=4$, closure forces $e_3=-e_1,e_4=-e_2$. In a convex nonflat
polygon the edge angles are strictly monotone over a total turn of
$\pm2\pi$, so each projective
direction occurs at most twice. Alternation therefore forces four
edges and a parallelogram.
\end{proof}

See \cref{fig:quadratic-geometry}.
\FloatBarrier
\begin{figure}[!htbp]
\centering
\begin{minipage}[t]{.47\textwidth}
\centering
\begin{tikzpicture}[x=.82cm,y=.82cm,line cap=round,line join=round]
 \node[panel title] at (-2.9,2.65) {(a) Nonzero trace-free form};
 \coordinate (o) at (-.9,.8);
 \draw[construction line] (o) circle (1.35);
 \draw[selected branch] ($(o)+(-45:1.65)$)--($(o)+(135:1.65)$);
 \draw[selected branch] ($(o)+(45:1.65)$)--($(o)+(225:1.65)$);
 \node[figure point] at (o) {};
 \node[figure label] at (.35,2.05) {$L$};
 \node[figure label] at (.35,-.45) {$L^\perp$};
 \node[figure label,align=left] at (1.25,-1.0)
   {$A\ne0,\quad \Phi_A(e_t)=\lambda_A\cos2(t-t_0)$\\
    projective zero directions: $\{L,L^\perp\}$};
\end{tikzpicture}
\end{minipage}\hfill
\begin{minipage}[t]{.47\textwidth}
\centering
\begin{tikzpicture}[x=.82cm,y=.82cm,line cap=round,line join=round]
 \node[panel title] at (-2.9,2.65) {(b) Nondegenerate triangle};
 \coordinate (A) at (-1.8,-.3);
 \coordinate (B) at (1.8,-.3);
 \coordinate (C) at (.45,2.2);
 \draw[geometric boundary,fill=modelblue] (A)--(B)--(C)--cycle;
 \coordinate (I) at (.261,.618);
 \draw[construction line,-{Stealth[length=2mm]}] (A)--($(A)!.48!(I)$)
   node[midway,below,font=\scriptsize] {$b_1$};
 \draw[construction line,-{Stealth[length=2mm]}] (B)--($(B)!.48!(I)$)
   node[midway,below,font=\scriptsize] {$b_2$};
 \draw[construction line,-{Stealth[length=2mm]}] (C)--($(C)!.48!(I)$)
   node[midway,right,font=\scriptsize] {$b_3$};
 \node[figure label,align=center] at (0,-.78)
   {three directions $\Longrightarrow A=0$};
\end{tikzpicture}
\end{minipage}

\vspace{1ex}
\begin{minipage}[t]{.47\textwidth}
\centering
\begin{tikzpicture}[x=.82cm,y=.82cm,line cap=round,line join=round]
 \node[panel title] at (-2.9,2.65) {(c) Parallelogram};
 \coordinate (A) at (-1.8,-.25);
 \coordinate (B) at (1.25,-.25);
 \coordinate (D) at (-.65,1.9);
 \coordinate (C) at (2.4,1.9);
 \draw[geometric boundary,fill=modelblue] (A)--(B)--(C)--(D)--cycle;
 \draw[map arrow] (A)--++(0:1.0) node[midway,below,font=\scriptsize] {$u$};
 \draw[map arrow] (A)--++(61.87:1.0) node[midway,left,font=\scriptsize] {$v$};
 \draw[construction line,-{Stealth[length=2mm]}] (A)--++(30.94:1.25)
   node[near end,above,font=\scriptsize] {$b_+$};
 \draw[construction line,-{Stealth[length=2mm]}] (B)--++(120.94:1.15)
   node[near end,right,font=\scriptsize] {$b_-$};
 \node[figure label,align=center] at (.3,-.82)
   {$b_+\parallel u+v,\quad b_-\parallel v-u$\\
    $b_+\perp b_-$};
\end{tikzpicture}
\end{minipage}\hfill
\begin{minipage}[t]{.47\textwidth}
\centering
\begin{tikzpicture}[x=.82cm,y=.82cm,line cap=round,line join=round]
 \node[panel title] at (-2.9,2.65) {(d) Orthogonal polygon};
 \path[draw=black,line width=.9pt,fill=modelblue]
   (-1.8,-.35)--(2.3,-.35)--(2.3,.55)--(.15,.55)--(.15,2.2)
   --(-1.8,2.2)--cycle;
 \foreach \x/\y/\ang in {-1.8/-.35/45,2.3/-.35/135,2.3/.55/225,.15/.55/225,.15/2.2/225,-1.8/2.2/315}{
   \draw[construction line] (\x,\y)--++(\ang:.58);
 }
 \node[figure label,align=center] at (.25,-.85)
   {two diagonal directions};
\end{tikzpicture}
\end{minipage}
\caption[Geometry of the quadratic classification]{Quadratic cancellation
geometry for \cref{cor:quadratic-dimension}.  A nonzero trace-free form has one
orthogonal pair of projective zero directions.  A triangle has three distinct
bisector directions and forces $A=0$.  A parallelogram has the orthogonal
bisectors $u+v$ and $v-u$ for unit side directions $u,v$, while every bisector
of an axis-parallel orthogonal polygon is diagonal.}
\label{fig:quadratic-geometry}
\end{figure}
\FloatBarrier

\begin{proposition}
\label{prop:nonquadratic-abundance}
For every bounded finitely cornered piecewise-$C^{1,\mathrm{Dini}}$
domain $\Omega$, the positively two-homogeneous functions in
$C^\infty(\R^2\setminus\{0\})\cap C^1(\R^2)$ satisfying Newton-kernel
signed tangent-model cancellation form an infinite-dimensional vector
space containing explicit nonquadratic functions. All satisfy the
density and PDE inequalities of
\cref{thm:curved-density,thm:pde-mazya}.
\end{proposition}

\begin{proof}
When vertices are present, list their boundary-ray directions modulo
$\pi$ as $\alpha_1,\ldots,\alpha_d$, where
$2\leq d\leq2|\mathcal V(\Omega)|$. For smooth $\pi$-periodic $G$, set
\[
 P(t)=\prod_{j=1}^d\sin^2(t-\alpha_j),\qquad F=(PG)',\qquad
 \Phi_F(re_t)=r^2F(t),\quad \Phi_F(0)=0.
\]
Periodicity cancels every half-circle, and $PG=0$ at every vertex
interval endpoint cancels each cone. The sign conditions coincide
because $F$ is $\pi$-periodic. The map $G\mapsto(PG)'$ is injective:
a constant $PG$ vanishes at a zero of $P$, hence $G=0$ by continuity.
Taking $G=1$ gives a nonzero mode of frequency $2d\geq4$, so a
nonquadratic example. Without vertices, use all smooth $\pi$-periodic
zero-mean $F$, including $\cos4t$. In either case
$|\Phi_F(\xi)|\leq\|F\|_\infty|\xi|^2$ and
$|\nabla\Phi_F(\xi)|\leq(2\|F\|_\infty+\|F'\|_\infty)|\xi|$,
which give the $C^1$ extension at zero. Apply the two criteria.
\end{proof}

\begin{corollary}\label{cor:regular-polygon-nonquadratic}
For a regular $N$-gon $P_N$, $N\geq3$, and any $t_0\in\R$, let
\[
 q_N=N/\gcd(N,2),\qquad
 \Phi_N(re_t)=r^2\cos\bigl(2q_N(t-t_0)\bigr),\quad \Phi_N(0)=0.
\]
This is a nonquadratic function in
$C^\infty(\R^2\setminus\{0\})\cap C^1(\R^2)$ satisfying signed
Newton-kernel cancellation and both the density and PDE inequalities
on $P_N$. In particular, an equilateral triangle admits
$\Phi_3(re_t)=r^2\cos6t$, although its admissible quadratic form is zero.
\end{corollary}

\begin{proof}
The plane, half-plane and vertex intervals contain respectively
$2q_N$, $q_N$ and $(N-2)/\gcd(N,2)$ periods of this zero-mean angular
trace. Evenness gives both signs, and $2q_N\geq4$ excludes quadratic
forms. Regularity follows as above; apply the two criteria and
\cref{cor:quadratic-rigidity}.
\end{proof}

\begin{remark}
For $K_\alpha(re_t)=r^{\alpha-2}e_t$, $0<\alpha<2$, the same angular
constructions with $\Phi(re_t)=r^pF(t)$, $p=2/(2-\alpha)$, satisfy
signed cancellation and the density estimate. They are $C^1$ at zero
because $|\nabla\Phi(\xi)|\lesssim_{p,F}|\xi|^{p-1}$.
For any odd angular kernel, both full-plane sign conditions already
follow from both sign conditions on one half-plane: the opposite
half-plane with sign $\varsigma$ has the same moment as the original
with sign $-\varsigma$.
\end{remark}

\section{Scope and further questions}
\label{sec:scope}
The cumulative atomic theorem applies to arbitrary nonempty bounded
open sets under signed plane cancellation and
$\mathcal B_{K,\Phi}(\Omega)<\infty$. It is a sufficient condition,
not a claimed characterization. On finitely cornered Dini domains,
tangent-model cancellation verifies this budget and is also necessary.
The same explicit strictly convex non-Dini $C^1$ domain supports the
density and PDE estimates for $\Phi_+$ but not $\Phi_-$. Ordinary
Dini regularity therefore cannot be necessary for each fixed
nonzero pair $(\Omega,\Phi)$.

For quadratic Newton integrands, \cref{thm:quad-cumulative-criterion}
does give a necessary-and-sufficient signed cumulative condition.
It uses all source centres and permits interscale cancellation.
For general homogeneous nonlinearities,
\cref{prop:atomic-necessary-primitive} gives a necessary bound on signed
truncated point-source integrals. The present paper does not derive
the general signed converse: its quadratic proof uses a bilinear
expansion that is unavailable for general $\Phi$. Its signed
necessary-and-sufficient conclusion is therefore confined here to
quadratic integrands. The absolute ring-by-ring summability in
$\mathcal B_{K,\Phi}$ is a
different sufficient condition.

The Dini extension allows finitely many holes. The additional non-Dini
extension proved here uses strict convexity and $C^1$ regularity,
together with classical $L^2$ Neumann layer solvability; it is not
an extension theorem on arbitrary Lipschitz domains. Neither extension
asserts a quantitative Dirichlet-energy or full-Hessian estimate.
The local strong-field closure requires an $L^p$ neighborhood of the
integration domain's closure, or global $L^p$ control in its global
branch. These hypotheses are supplied by qualitative finite energy
in the Newton applications.

$N_\Omega(u)$ cannot control the full Hessian measure, Dirichlet
energy, or centered supremum norm. For a ball
$B(a,2r)\Subset\Omega$, a cutoff $\eta$ equal to one on $B(a,r)$,
and a radial nonnegative unit-mass mollifier $\rho$, the tests
$u_\varepsilon=\eta(\Gamma*\rho_\varepsilon)(\cdot-a)$ have bounded
$N_\Omega$ while
\[
 \|D^2u_\varepsilon\|_{L^1(\Omega)}\asymp\log(1/\varepsilon),\qquad
 \|\nabla u_\varepsilon\|_{L^2(\Omega)}^2
 \asymp\|u_\varepsilon-(u_\varepsilon)_\Omega\|_{L^\infty(\Omega)}
 \asymp\log(1/\varepsilon).
\]
This follows directly from the Newton potential on
$\varepsilon<|x-a|<r$ and scaling on the core; cutoff terms are
uniformly bounded. For related full-Hessian embeddings and
$L^1$ non-inequalities see
\cite{AdamsFournier,Demengel,Ornstein,MazyaShaposhnikova}.

The infinite-sector theorem requires compactly supported mean-zero
data for ordinary absolute Lebesgue integrability. The compact sector
extension is linear for each fixed support radius; compatibility
between different radii is not asserted. In higher dimensions, the
fixed-set moment argument persists, while a geometric verification
must organize edges and higher-codimension strata. Cusps, accumulating
corners, and nonunique tangent models are not classified by the present
geometric corollaries.
\appendix
\section{\bf A direct proof in the quadratic and superquadratic range}
\label{app:superquadratic-mixed}

This appendix gives a short unified proof of the mixed estimate for
$p\geq2$. It develops the positive-kernel overlap argument in
\cite[proof of Theorem~2.3 for $p=2$, equations~(3.6)--(3.9), and Lemma~6.2]
{StolyarovWhole}; the modifications for $p>2$ are already discussed in
\cite[Section~5, after equation~(5.4)]{StolyarovWhole}.
The proof uses Jensen's inequality and exact separation of the radial
shells. Angular boundedness suffices by domination with the positive
scalar kernel $|x|^{-2/p}$.

Put
\[
 q:=2-\alpha=\frac2p\in(0,1],
 \qquad \Lambda_K:=\|\wtK\|_{L^\infty(\Sph^1)}.
\]
For $r>0$, introduce the scalar radial kernels
\[
 U_r(x):=
 \begin{cases}
  |x|^{-q},&|x|\geq r,\\
  0,&|x|<r,
 \end{cases}
 \qquad
 V_r(x):=
 \begin{cases}
  |x|^{-q},&r/2\leq|x|<r,\\
  0,&\text{otherwise}.
 \end{cases}
\]

\begin{lemma}
\label{lem:superquadratic-shell-overlap}
For $a\in\R^2$ and $r>0$, set
\[
 \Gamma_r(a):=\int_{\R^2}
 \Big(U_r(x)^{p-1}V_r(x-a)
      +U_r(x)V_r(x-a)^{p-1}\Big)\,\dd x.
\]
There is $C_p<\infty$ such that
\begin{equation}\label{eq:superquadratic-shell-profile}
 \Gamma_r(a)\leq C_p\vartheta\!\left(\frac{|a|}{r}\right),
 \qquad
 \vartheta(t):=
 \begin{cases}
  t,&0\leq t\leq\frac14,\\
  1,&\frac14<t<4,\\
  t^{-q},&t\geq4.
 \end{cases}
\end{equation}
Consequently,
\begin{equation}\label{eq:superquadratic-shell-sum}
 \sup_{a\in\R^2}
 \sum_{j\in\mathbb Z}\Gamma_{2^{-j-1}}(a)<\infty.
\end{equation}
The constants depend only on $p$.
\end{lemma}

\begin{proof}
Write $A=|a|$. At $A=0$ the two sharp radial regions are disjoint, so
$\Gamma_r(0)=0$. For $A>0$ the integrands are supported on
$D_{r,a}=\{|x|\geq r,\ r/2\leq|x-a|<r\}$.
If $r\geq4A$, then $r\leq|x|<r+A$, both weights are comparable to
$r^{-q}$, and $|D_{r,a}|\lesssim rA$; hence
$\Gamma_r(a)\lesssim_p A/r$.
If $r\leq A/4$, then $|x|\simeq A$ on $D_{r,a}$ and its area is
$O(r^2)$, giving
\[
 \Gamma_r(a)\lesssim_p
 (r/A)^{q(p-1)}+(r/A)^q\lesssim_p(r/A)^q.
\]
For $A/4<r<4A$, all weights are $O_p(r^{-q})$ and the area is
$O(r^2)$, so $\Gamma_r(a)\lesssim_p1$.
These are \eqref{eq:superquadratic-shell-profile}. Summing the coarse
and fine ranges over dyadic $r$ gives two geometric series, while
only boundedly many radii lie in the middle range. This proves
\eqref{eq:superquadratic-shell-sum}.
\end{proof}

\begin{theorem}
\label{thm:superquadratic-mixed-direct}
Let $2\leq p<\infty$, put $q=2/p$, let $m\geq1$, and suppose
\[
 K(x)=|x|^{-q}\wtK(x/|x|)\quad(x\ne0),\qquad
 \wtK\in L^\infty(\Sph^1;\R^m).
\]
Choose a measurable representative of $\wtK$ satisfying
$|\wtK(\omega)|\leq\|\wtK\|_{L^\infty(\Sph^1)}$ for every
$\omega\in\Sph^1$, set $K(0)=0$, define the sharp annular truncations by
\eqref{eq:annular-kernels}, and set
$\Lambda_K=\|\wtK\|_{L^\infty(\Sph^1)}$. Then every real-valued
$h\in L^1(\R^2)$ satisfies the stronger all-scale estimate
\begin{equation}\label{eq:superquadratic-mixed-all-scales}
 \sum_{j\in\mathbb Z}\int_{\R^2}
 \Mp\bigl(|K_{\leq j}*h|,|K_{j+1}*h|\bigr)\,\dd x
 \leq C_p\Lambda_K^p\|h\|_1^p.
\end{equation}
In particular, the constant is independent of $h$, $\Phi$, the scale, and
all domain parameters. No cancellation, parity, or spherical-moment
condition is required. Among the angular assumptions on $K$, the proof
uses only $\wtK\in L^\infty(\Sph^1)$; homogeneity and the sharp
annular truncation are retained.
\end{theorem}

\begin{proof}
Let $\mu=|h|\,\mathcal L^2$, $M=\|h\|_1$, and $r_j=2^{-j-1}$;
assume $M>0$. The bounded sharp truncations satisfy
\begin{equation}\label{eq:superquadratic-positive-majorants}
 |K_{\leq j}*h|\leq\Lambda_K U_{r_j}*\mu,\qquad
 |K_{j+1}*h|\leq\Lambda_K V_{r_j}*\mu.
\end{equation}
For $p\geq2$, including $p=2$, the quantity $\Mp(s,t)$ is
$\tfrac12(s^{p-1}t+st^{p-1})$. Jensen's inequality gives, for
nonnegative bounded $W$,
\begin{equation}\label{eq:superquadratic-jensen}
 (W*\mu)^{p-1}\leq M^{p-2}(W^{p-1}*\mu).
\end{equation}
Apply it to $U_{r_j},V_{r_j}$, expand their products, and use Tonelli:
\[
 \begin{aligned}
 &\sum_j\int_{\R^2}\Mp(|K_{\leq j}*h|,|K_{j+1}*h|)\\
 &\quad\leq\frac{\Lambda_K^pM^{p-2}}2
  \iint\sum_j\Gamma_{r_j}(z-y)\,\dd\mu(y)\,\dd\mu(z)
 \leq C_p\Lambda_K^pM^p,
 \end{aligned}
\]
where \cref{lem:superquadratic-shell-overlap} gives the last bound.
This proves the all-scale estimate; nonnegativity also permits any
threshold restriction, including \eqref{eq:threshold-mixed}.
\end{proof}

\begin{remark}
The exact separation of the annular regions in \eqref{eq:annular-kernels} is essential
to this short proof: it gives $\Gamma_r(0)=0$.  For overlapping smooth
cutoffs the corresponding same-centre interaction need not vanish, so the
argument above does not apply without an additional idea.
\end{remark}

\section*{Data Availability}
Data sharing is not applicable to this article as no datasets were
generated or analyzed during the current study.

\section*{Declarations}

\noindent\textbf{Funding.}
W.~Zou is funded by the National Key R\&D Program of China\newline
(Grant 2023YFA1010001),
NSFC (12571123) and NSFC(12671139).

\smallskip
\noindent\textbf{Statements and Declarations.}
The authors have no relevant financial or non-financial interests to disclose.

\smallskip
\noindent\textbf{Conflict of Interest Statement.}
The authors declare that they have no conflict of interest.

\section*{Acknowledgements}

AI tools (DeepSeek and ChatGPT) assisted with figures, language, reference searches, proof reorganization, non-Dini examples, and intermediate proof checks. The authors retain responsibility for the mathematical arguments, references, and conclusions. This disclosure records assistance with arguments as well as exposition.

\microtypesetup{expansion=false}
\renewcommand{\bibliofont}{\fontsize{9}{10.5}\selectfont}


\begin{thebibliography}{99}

\bibitem{AdamsFournier}
R.~A. Adams and J.~J.~F. Fournier,
\emph{Sobolev Spaces}, 2nd ed.,
Pure and Applied Mathematics, vol.~140, Academic Press, Amsterdam, 2003.

\bibitem{BourgainBrezis}
J.~Bourgain and H.~Brezis,
\emph{On the equation $\operatorname{div}Y=f$ and application to control of phases},
J. Amer. Math. Soc. \textbf{16} (2003), no.~2, 393--426,
\href{https://doi.org/10.1090/S0894-0347-02-00411-3}{doi:10.1090/S0894-0347-02-00411-3}.

\bibitem{Demengel}
F.~Demengel,
\emph{Fonctions \`a hessien born\'e},
Ann. Inst. Fourier (Grenoble) \textbf{34} (1984), no.~2, 155--190.

\bibitem{GmeinederRaitaDomains}
F.~Gmeineder and B.~Rai{\c t}\u{a},
\emph{Embeddings for $\mathbb A$-weakly differentiable functions on domains},
J. Funct. Anal. \textbf{277} (2019), no.~12, Article 108278, 33~pp.,
\href{https://doi.org/10.1016/j.jfa.2019.108278}{doi:10.1016/j.jfa.2019.108278}.

\bibitem{GmeinederRaitaVanSchaftingen}
F.~Gmeineder, B.~Rai{\c t}\u{a}, and J.~Van Schaftingen,
\emph{Boundary ellipticity and limiting $L^1$-estimates on halfspaces},
Adv. Math. \textbf{439} (2024), Paper No.~109490, 25~pp.,
\href{https://doi.org/10.1016/j.aim.2024.109490}{doi:10.1016/j.aim.2024.109490}.

\bibitem{Grisvard}
P.~Grisvard,
\emph{Elliptic Problems in Nonsmooth Domains},
Monographs and Studies in Mathematics, vol.~24, Pitman, Boston, MA, 1985.

\bibitem{Kondratev1967}
V.~A. Kondrat'ev,
\emph{Boundary value problems for elliptic equations in domains with conical or angular points},
Trans. Moscow Math. Soc. \textbf{16} (1967), 227--313
(English translation of Trudy Moskov. Mat. Obshch. \textbf{16} (1967), 209--292),
\href{https://www.mathnet.ru/eng/mmo186}{MathNet: mmo186}.


\bibitem{KralMedkova1998}
J.~Kr\'al and D.~Medkov\'a,
\emph{Essential norms of a potential theoretic boundary integral operator
in $L^1$},
Math. Bohem. \textbf{123} (1998), no.~4, 419--436,
\href{https://doi.org/10.21136/MB.1998.125966}{doi:10.21136/MB.1998.125966}.

\bibitem{KozlovMazyaRossmann}
V.~A. Kozlov, V.~G. Maz'ya, and J.~Rossmann,
\emph{Elliptic Boundary Value Problems in Domains with Point Singularities},
Mathematical Surveys and Monographs, vol.~52, American Mathematical Society, Providence, RI, 1997,
\href{https://doi.org/10.1090/surv/052}{doi:10.1090/surv/052}.

\bibitem{Mazya2010}
V.~G. Maz'ya,
\emph{Estimates for differential operators of vector analysis involving $L^1$-norm},
J. Eur. Math. Soc. (JEMS) \textbf{12} (2010), no.~1, 221--240,
\href{https://doi.org/10.4171/JEMS/195}{doi:10.4171/JEMS/195}.

\bibitem{MazyaProblems}
V.~Maz'ya,
\emph{Seventy five (thousand) unsolved problems in analysis and partial differential equations},
Integral Equations Operator Theory \textbf{90} (2018), no.~2, Paper No.~25, 44~pp.

\bibitem{MazyaShaposhnikova}
V.~Maz'ya and T.~Shaposhnikova,
\emph{A collection of sharp dilation invariant integral inequalities for differentiable functions},
in \emph{Sobolev Spaces in Mathematics I: Sobolev Type Inequalities},
International Mathematical Series, vol.~8, Springer, New York, 2009, pp.~223--247,
\href{https://doi.org/10.1007/978-0-387-85648-3_8}{doi:10.1007/978-0-387-85648-3\_8}.

\bibitem{MedkovaNeumann1998}
D.~Medkov\'a,
\emph{Solution of the Neumann problem for the Laplace equation},
Czechoslovak Math. J. \textbf{48} (1998), no.~4, 763--784,
\href{https://doi.org/10.1023/A:1022447908645}{doi:10.1023/A:1022447908645}.

\bibitem{Ornstein}
D.~Ornstein,
\emph{A non-inequality for differential operators in the $L^1$ norm},
Arch. Rational Mech. Anal. \textbf{11} (1962), 40--49.

\bibitem{Spector}
D.~Spector,
\emph{New directions in harmonic analysis on $L^1$},
Nonlinear Anal. \textbf{192} (2020), Paper No.~111685, 20~pp.,
\href{https://doi.org/10.1016/j.na.2019.111685}{doi:10.1016/j.na.2019.111685}.

\bibitem{StolyarovWhole}
D.~Stolyarov,
\emph{Fractional integration of summable functions: Maz'ya's $\Phi$-inequalities},
Ann. Sc. Norm. Super. Pisa Cl. Sci. (5) \textbf{25} (2024), no.~3, 1727--1752,
\href{https://doi.org/10.2422/2036-2145.202110_001}{doi:10.2422/2036-2145.202110\_001}.
Author preprint: \href{https://arxiv.org/abs/2109.08014}{arXiv:2109.08014}.

\bibitem{StolyarovDomains}
D.~Stolyarov,
\emph{Maz'ya's $\Phi$-inequalities on domains},
Zap. Nauchn. Sem. POMI \textbf{537} (2024), 128--150; English translation:
J. Math. Sci. \textbf{295} (2025), no.~4, 458--472,
\href{https://doi.org/10.1007/s10958-026-08192-x}{doi:10.1007/s10958-026-08192-x}.
Author preprint: \href{https://arxiv.org/abs/2407.14052}{arXiv:2407.14052}.

\bibitem{VanSchaftingen2008}
J.~Van Schaftingen,
\emph{Estimates for $L^1$-vector fields under higher-order differential conditions},
J. Eur. Math. Soc. \textbf{10} (2008), no.~4, 867--882,
\href{https://doi.org/10.4171/JEMS/133}{doi:10.4171/JEMS/133}.

\bibitem{VanSchaftingen2013}
J.~Van Schaftingen,
\emph{Limiting Sobolev inequalities for vector fields and canceling linear differential operators},
J. Eur. Math. Soc. \textbf{15} (2013), no.~3, 877--921,
\href{https://doi.org/10.4171/JEMS/380}{doi:10.4171/JEMS/380}.


\bibitem{Verchota1984}
G.~Verchota,
\emph{Layer potentials and regularity for the Dirichlet problem for
Laplace's equation in Lipschitz domains},
J. Funct. Anal. \textbf{59} (1984), no.~3, 572--611,
\href{https://doi.org/10.1016/0022-1236(84)90066-1}{doi:10.1016/0022-1236(84)90066-1}.

\end{thebibliography}
\end{document}